\documentclass[notitlepage, 12pt, reqno]{amsart}

\pdfoutput=1

\usepackage{graphicx}
\usepackage{amssymb}
\usepackage[foot]{amsaddr}
\usepackage{hyperref}
\usepackage{amsmath}
\usepackage{enumitem}
\usepackage{cite}
\usepackage{mathtools}
\usepackage{verbatim}
\usepackage{subcaption}

\usepackage{amsthm}

\usepackage{comment}
\usepackage{setspace}

\usepackage{tikz}
\usetikzlibrary{arrows.meta} 
\usetikzlibrary{decorations.pathmorphing} 

\usepackage{pgfplots}
\pgfplotsset{compat=newest}

\usepackage{color}
\usepackage[letterpaper,margin=1in]{geometry}

\usepackage{enumitem} 
\newlist{condenum}{enumerate}{1} 
\setlist[condenum]{label=\bfseries Condition \arabic*.,
                   ref=\arabic*, wide}

\numberwithin{equation}{section}

\newtheorem{theorem}{Theorem}[section]

\theoremstyle{plain}
\newtheorem{corollary}[theorem]{Corollary}
\newtheorem{lemma}[theorem]{Lemma}
\newtheorem{proposition}[theorem]{Proposition}

\theoremstyle{definition}
\newtheorem{remark}[theorem]{Remark}
\newtheorem{assumption}[theorem]{Assumption}

\newcommand{\isset}[1]{\mathbb{#1}} 
\newcommand{\ismat}[1]{\mathbf{#1}} 
\newcommand{\isvec}[1]{\mathbf{#1}} 
\newcommand{\genf}[1]{\mathcal{#1}} 

\newcommand{\arrowset}{\isset{X}} 
\newcommand{\objectset}{\isset{O}} 
\newcommand{\groupoid}{\isset{G}} 
\newcommand{\genset}[1]{\isset{A}_{#1}} 
\newcommand{\source}{\ensuremath{\mathfrak{s}}} 
\newcommand{\target}{\ensuremath{\mathfrak{t}}} 

\newcommand{\matinds}[3]{{#1}_{#2, #3}} 
\newcommand{\matindsbar}[3]{\overline{{#1}}_{#2, #3}} 
\newcommand{\vecinds}[2]{{#1}_#2} 
\newcommand{\vecindstilde}[2]{\tilde{{#1}}_#2} 

\newcommand{\metric}{L} 
\newcommand{\metrictwo}{F} 
\newcommand{\defn}{\coloneqq} 
\newcommand{\st}{\,|\,}  
\newcommand{\indic}{\mathbb{I}} 
\newcommand{\transpose}{\ensuremath{^T}} 
\newcommand{\idmat}{\ismat{I}} 
\newcommand{\zeromat}{\ismat{0}} 
\newcommand{\onevec}{\ensuremath{\isvec{1}}} 
\newcommand{\zerovec}{\ensuremath{\isvec{0}}} 

\let\oldleft\left
\let\oldright\right
\renewcommand{\left}{\mathopen{}\mathclose\bgroup\oldleft}
\renewcommand{\right}{\aftergroup\egroup\oldright}

\begin{document}

\title{On a random entanglement problem}

\author{Gage Bonner}
\email{gbonner@wisc.edu}

\author{Jean-Luc Thiffeault}
\email{jeanluc@math.wisc.edu}

\author{Benedek Valk{\'o}}
\email{valko@math.wisc.edu}

\address{Department of Mathematics, University of Wisconsin --
	Madison, WI 53706, USA}


\date{}

\begin{abstract}
We study a model for the entanglement of a two-dimensional reflecting Brownian motion in a bounded region divided into two halves by a wall with three or more small windows. We map the Brownian motion into a Markov Chain on the fundamental groupoid of the region. We quantify entanglement of the path with the length of the appropriate element in this groupoid. Our main results are a law of large numbers and a central limit theorem for this quantity. The constants appearing in the limit theorems are expressed in terms of a coupled system of quadratic equations.
\end{abstract}


\maketitle


\section{Motivation}
\label{sec:motivation}

We consider a reflecting Brownian motion in a piecewise smooth bounded region of $\isset{R}^2$. This region is divided into a top and bottom half by a wall punctured by $N \geq 3$ small windows, as shown in Figure~\ref{fig:4_window_arena_diagram}.

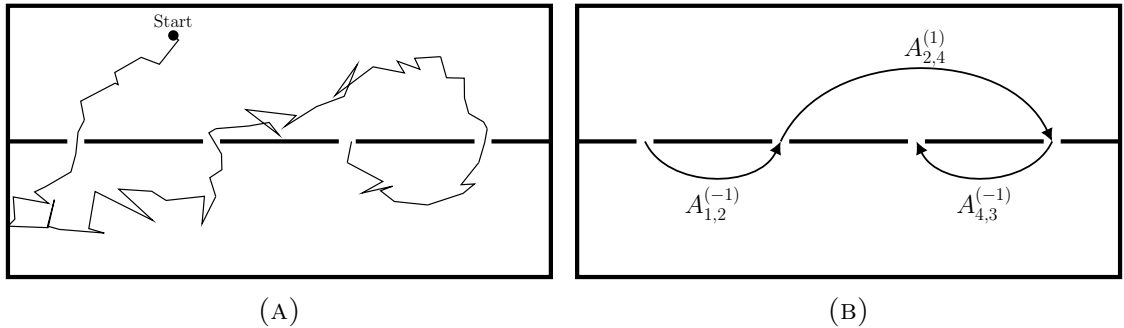
\begin{figure}[ht]
     \centering
     \begin{subfigure}[t]{0.45\linewidth}
        \centering
        \resizebox{\linewidth}{!}{
        \begin{tikzpicture}
        \pgfmathsetseed{3}

        \coordinate (TL) at (-64mm,32mm);
        \coordinate (TR) at (64mm,32mm);
        \coordinate (BL) at (-64mm,-32mm);
        \coordinate (BR) at (64mm,-32mm);

        \draw[black, line width=1mm] (BL) rectangle (TR);

        \coordinate (ML) at (-64mm,0mm);
        \coordinate (MR) at (64mm,0mm);
        \coordinate (w1L) at (-50mm,0mm);
        \coordinate (w1R) at (-46mm,0mm);
        \coordinate (w2L) at (-18mm,0mm);
        \coordinate (w2R) at (-14mm,0mm);
        \coordinate (w3L) at (14mm,0mm);
        \coordinate (w3R) at (18mm,0mm);
        \coordinate (w4L) at (46mm,0mm);
        \coordinate (w4R) at (50mm,0mm);

        \draw[-, black, line width=1mm] (ML) to (w1L);
        \draw[-, black, line width=1mm] (w1R) to (w2L);
        \draw[-, black, line width=1mm] (w2R) to (w3L);
        \draw[-, black, line width=1mm] (w3R) to (w4L);
        \draw[-, black, line width=1mm] (w4R) to (MR);

        \node[] at (-25mm, 28.5mm) {Start};
        \node[circle, fill=black,inner sep=0pt,minimum size=7pt] at (-25mm, 25mm) {};

        \draw[-, black, thick, decoration={random steps, amplitude=4mm}, decorate] (-25mm, 25mm) to (-46mm, 5mm);
        \draw[-, black, thick, decoration={random steps, amplitude=2mm}, decorate]  (-46mm, 5mm) to (-48mm, -3mm);
        \draw[-, black, thick, decoration={random steps, amplitude=4mm}, decorate] (-48mm, -3mm) to (-64mm, -20mm);
        \draw[-, black, thick, decoration={random steps, amplitude=6mm}, decorate] (-64mm, -20mm) to (-15mm, -13mm);
        \draw[-, black, thick, decoration={random steps, amplitude=2mm}, decorate] (-15mm, -13mm) to (-15mm, 3mm);
        \draw[-, black, thick, decoration={random steps, amplitude=6mm}, decorate] (-15mm, 3mm) to (38mm, 20mm);
        \draw[-, black, thick, decoration={random steps, amplitude=2mm}, decorate] (38mm, 20mm) to (49mm, 3mm);
        \draw[-, black, thick, decoration={random steps, amplitude=2mm}, decorate] (49mm, 3mm) to (47mm, -5mm);
        \draw[-, black, thick, decoration={random steps, amplitude=4mm}, decorate] (47mm, -5mm) to (35mm, -15mm);
        \draw[-, black, thick, decoration={random steps, amplitude=6mm}, decorate] (35mm, -15mm) to (16mm, -5mm);
        \draw[-, black, thick, decoration={random steps, amplitude=4mm}, decorate] (16mm, -5mm) to (17mm, 0mm);
        \end{tikzpicture}
        }
        \caption{}
        \label{fig:path}
     \end{subfigure}
     \begin{subfigure}[t]{0.45\linewidth}
        \centering
        \resizebox{\linewidth}{!}{
        \begin{tikzpicture}
        \pgfmathsetseed{3}

        \coordinate (TL) at (-64mm,32mm);
        \coordinate (TR) at (64mm,32mm);
        \coordinate (BL) at (-64mm,-32mm);
        \coordinate (BR) at (64mm,-32mm);

        \draw[black, line width=1mm] (BL) rectangle (TR);

        \coordinate (ML) at (-64mm,0mm);
        \coordinate (MR) at (64mm,0mm);
        \coordinate (w1L) at (-50mm,0mm);
        \coordinate (w1M) at (-48mm,0mm);
        \coordinate (w1R) at (-46mm,0mm);
        \coordinate (w2L) at (-18mm,0mm);
        \coordinate (w2M) at (-16mm,0mm);
        \coordinate (w2R) at (-14mm,0mm);
        \coordinate (w3L) at (14mm,0mm);
        \coordinate (w3M) at (16mm,0mm);
        \coordinate (w3R) at (18mm,0mm);
        \coordinate (w4L) at (46mm,0mm);
        \coordinate (w4M) at (48mm,0mm);
        \coordinate (w4R) at (50mm,0mm);

        \draw[-, black, line width=1mm] (ML) to (w1L);
        \draw[-, black, line width=1mm] (w1R) to (w2L);
        \draw[-, black, line width=1mm] (w2R) to (w3L);
        \draw[-, black, line width=1mm] (w3R) to (w4L);
        \draw[-, black, line width=1mm] (w4R) to (MR);


        \draw[-{Latex[length=3mm, width=3mm]}, black, very thick] (w1M) to[out=-65,in=-115] (w2M);
        \draw[-{Latex[length=3mm, width=3mm]}, black, very thick] (w2M) to[out=65,in=115] (w4M);
        \draw[-{Latex[length=3mm, width=3mm]}, black, very thick] (w4M) to[out=-115,in=-65] (w3M);

        \node[] at (-32mm, -15mm) {\Large $A_{1, 2}^{(-1)}$};
        \node[] at (18mm, 22mm) {\Large $A_{2, 4}^{(1)}$};
        \node[] at (32mm, -15mm) {\Large $A_{4, 3}^{(-1)}$};

        \end{tikzpicture}
        }
        \caption{}
        \label{fig:arcs}
     \end{subfigure}

    \caption{(A) A planar domain with $N= 4$ windows and a sample Brownian path.  (B) The same path expressed as arcs labeled with their associated generating elements. (See Section \ref{sec:math_setup}.) }
    \label{fig:4_window_arena_diagram}
\end{figure}

The Brownian path winds around the wall segments through the windows, and becomes progressively more entangled. The entanglement can be quantified by mapping the path at a given time to an element of the fundamental groupoid~\cite{brown2006topology} of the region, and considering the length of that element in the appropriate sense (see Section \ref{sec:math_setup}). Our goal is to study the asymptotic growth of this length as a function of time. The growth rate of words in the groupoid serves as an indication for the nature of the growth one would expect to see in the winding problem. Motivated by random walks on free groups~\cite{Guivarch1980, Sawyer1987, Lalley1993} one expects the length in the fundamental groupoid to grow linearly in time. In the present paper we identify formulas for both the growth rate and the limiting fluctuations around the mean, in the setting involving \emph{small windows}. Our main contribution (Theorem~\ref{thm:main}) is the proof of these limit theorems in a general setting.  The limits are described in terms of a set of coupled quadratic equations, which can be readily solved numerically.

One can follow the entanglement of the Brownian path by observing it at times when it visits a new window, and considering the length of the corresponding element in the fundamental groupoid.  This fundamental groupoid is generated by equivalence classes of oriented paths connecting two windows, with each such path lying in the upper or lower half of the plane. Between successive observations, the groupoid element corresponding to the path is appended by a random generating element whose distribution depends on the location inside the window. Motivated by the narrow escape problem~\cite{HolcmanSchuss}, as the windows shrink in size this location dependence disappears, and we arrive at a Markov chain on the fundamental groupoid. Our limit theorems are about the length of the groupoid element in this Markov chain.

Probabilistic winding problems on surfaces have a long history.  A classical example is the asymptotic behavior of the winding of a planar Brownian motion around a point.  Spitzer~\cite{Spitzer1958} showed that the  winding angle at time $t$, scaled by $\log t$, converges to a standard Cauchy distribution as time goes to infinity.  The fact that the limit distribution has no  moments  can be explained by  the large amount of winding that the Brownian path can pick up when it comes near the origin.  This model has been thoroughly investigated by many authors~\cite{Messulam1982, Pitman1984, Pitman1986, Pitman1989}.

When using Brownian motion to model, say, polymer entanglement~\cite{Grosberg2003}, it is more realistic to regularize the problem in some way.  This can be accomplished, for example, by replacing the punctual winding center by a finite topological disk~\cite{Grosberg2003, Geng2018_preprint}, by adding a persistence length to the motion~\cite{Tanaka2012}, or by considering a random walk instead~\cite{Rudnick1987, Rudnick1988, Berger1987, Berger1988, Belisle1989, Belisle1991}.  In the regularized problem the scaling limit for the winding angle becomes the
hyperbolic secant distribution, where all the moments exist.  Unsurprisingly, confinement to a finite region greatly increases the rate of winding, since the Brownian path returns near the winding center more frequently~\cite{Grosberg2003, Geng2018_preprint, Wen2019}.

A more challenging problem is the study of winding around multiple points or topological disks.  A natural first approach to this problem is the homological route, where  one examines the joint distribution of winding angles around each winding center. In the scaling limit these winding angles converge  to independent Cauchy distributions \cite{Pitman1984, Pitman1986}.  The homological route is inherently Abelian, in that the order of winding around the centers is lost.  Watanabe~\cite{Watanabe2000} studied winding on punctured surfaces of higher genus, and derived Gaussian limit distributions for the windings around each handle.

Another approach is via the fundamental group of the punctured surface, which is the group of deck transformations on its universal cover.  In that case, the non-Abelian aspect of the windings is captured, and we may regard distance in the universal cover as a measure of entanglement of the Brownian motion.  This approach was first introduced by It\^o and McKean \cite{ItoMcKean1965, McKean}
who considered the twice-punctured plane.  (See also \cite{McKean1984, Lyons1984}.)  Gruet~\cite{Gruet1998} finds that the length of the word at time $t$ in the fundamental group of the thrice-punctured sphere grows at least like $t \log t$ as $t \to \infty$.  Desenonges \cite{Desenonges2019} considers a similar problem on a wider class of surfaces with $n$ punctures.  See also the book by Nechaev \cite{Nechaev} for winding in an infinite lattice of points. Note that the region in our Brownian entanglement problem is topologically equivalent to a sphere with $N$ holes, hence our result belongs to this class of non-Abelian problems. Our Markov chain can also be considered as a random walk on a regular language (see Remark~\ref{remark:gilch}).

Our paper is organized as follows. Section~\ref{sec:math_setup} and Section~\ref{sec:main} contain the precise setup of the problem and our main result. In Section~\ref{sec:examples} we give a number of applications of our main theorem. Section~\ref{sec:outline_of_proof} provides the key steps of the proof, which are proved in the rest of the paper (Section~\ref{sec:proofs} and Appendix~\ref{sec:appendix_proofs}).

\section{Preliminaries}
\label{sec:math_setup}

\subsection{The fundamental groupoid \texorpdfstring{$\groupoid_N$}{G(N)}}

We consider the groupoid representing the homotopy classes of continuous paths that start and end at the midpoints of the windows as in Figure~\ref{fig:4_window_arena_diagram}.

Recall that a groupoid is defined by a set of `objects' $\objectset$ and `arrows' $\arrowset$ and the following functions:
\begin{enumerate}[label=(G\arabic{*})]
    \item There are functions $\source$ (source) and $\target$ (target) from $\arrowset\to \objectset$. \label{eq:g_prop_source_and_target}
    \item There is a composition function $(f_1,f_2)\to f_1 f_2$ on a subset of $\arrowset\times \arrowset$ which is defined for $f_1, f_2$ if $\target(f_1) = \source(f_2)$, and in that case $\source(f_1 f_2) = \source(f_1)$, $\target(f_1 f_2) = \target(f_2)$. The composition function is associative. \label{eq:g_prop_composition}
    \item For each $i\in \objectset$ there is a unique unit element $e_i\in \arrowset$ with $\source(e_i) = \target(e_i) = i$ for which $e_i f = f$, $f e_i = f$ whenever these are defined.  \label{eq:g_prop_identities}
    \item There is an inverse for each element of $\arrowset$ satisfying $\source(f^{-1}) = \target(f)$, $\target(f^{-1}) = \source(f)$ and $f f^{-1}=e_{\source(f)}$, $f^{-1} f=e_{\target(f)}$. \label{eq:g_prop_inverses}
\end{enumerate}
We consider a groupoid~$\groupoid_N$ with object set $\objectset_N=\objectset=\{1,2,\dots,N\}$, and arrow set $\arrowset_N=\arrowset$ generated by the elements in
\begin{equation} \label{eq:genset_def}
    \genset{N} \defn \{A_{i, j}^{(k)} : i \neq j , \ \ 1 \leq i,  j \leq N, \ \ k \in \{-1, 1\} \},
\end{equation}
with
\begin{equation}
\source(A_{i, j}^{(k)})=i, \qquad \target(A_{i, j}^{(k)})=j.
\end{equation}
For convenience we define~$A_{i, i}^{(k)} \defn e_{i}$.
The set of generating relations for our groupoid is given by the relations
\begin{equation} \label{eq:groupoid_letter_relations}
 A_{i, j}^{(k)}A_{j, \ell}^{(k)}=A_{i, \ell}^{(k)} \qquad\qquad i, j, \ell \in \objectset,\,  k \in \{-1, 1\}.
\end{equation}
We call a composition of finitely many generating elements a word. We include the unit elements as words, and call them empty words.
Each arrow in $\groupoid_N$ corresponds to an equivalence class of words.  Two words are in the same equivalence class if they can be transformed into each other by repeated application of relations \eqref{eq:groupoid_letter_relations}. We say that a nonempty word is reduced if we cannot apply the relations \eqref{eq:groupoid_letter_relations} to reduce the number of generators used in the word. We consider the empty words reduced as well.

For $i\neq j$, the arrow $A_{i,j}^{(+1)}$ (respectively, $A_{i,j}^{(-1)}$) corresponds to a simple path in the upper (respectively, lower) part of the region connecting window $i$ to window $j$ as in Figure~\ref{fig:arcs}.  Figure \ref{fig:letters_schematic} shows a schematic of the groupoid structure for $N = 3$.

The arrows of the groupoid~$\groupoid_N$ can be represented as equivalence classes of paths in the directed multi-graph  which has vertices $\objectset=\{1,2,\dots, N\}$ and directed edges of the form $(i,j,k)$ with $i\neq j \in \objectset$, $k\in \{-1,1\}$. The generating elements $A_{i,j}^{(k)}$ correspond to the directed edges $(i,j,k)$; a composition of generating elements (a word) corresponds to a path in the multi-graph. The  starting and ending vertices of a path are the results of the source and target functions. Two paths are in the same equivalence class (and correspond to the same arrow in $\arrowset$) if they can be transformed into each other by the repeated use of the following operations and their inverses
\begin{enumerate}[label=(EC\arabic{*})]
    \item Deleting a backtracking step  $(i,j,k), (j,i,k)$ \label{eq:eclass_backtrack}
    \item Replacing two consecutive steps $(i,j,k), (j,\ell,k)$ with $(i,\ell,k)$ if $i,j,\ell$ are different. \label{eq:eclass_same_plane}
\end{enumerate}
These operations correspond to the generating relations \eqref{eq:groupoid_letter_relations}.
A path corresponds to a reduced word if if we cannot use either of the moves \ref{eq:eclass_backtrack} and \ref{eq:eclass_same_plane} on it.
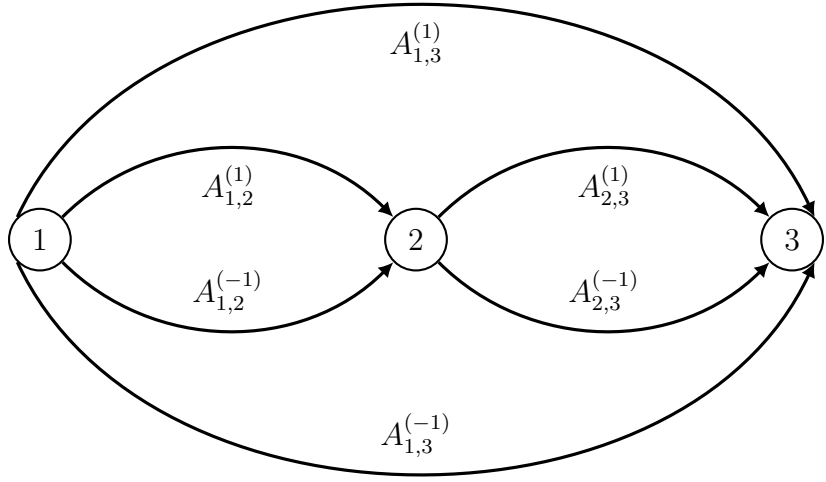
\begin{figure}[ht]
\centering
\begin{tikzpicture}

\coordinate (Ltr) at (-47mm,3mm);
\coordinate (Ltl) at (-53mm,3mm);

\coordinate (Lbr) at (-47mm,-3mm);
\coordinate (Lbl) at (-53mm,-3mm);

\coordinate (Mtl) at(-3mm,3mm);
\coordinate (Mtr) at(3mm,3mm);

\coordinate (Mbl) at(-3mm,-3mm);
\coordinate (Mbr) at(3mm,-3mm);

\coordinate (Rtl) at (47mm,3mm);
\coordinate (Rtr) at (53mm,3mm);

\coordinate (Rbl) at (47mm,-3mm);
\coordinate (Rbr) at (53mm,-3mm);

\node[] at (-50mm,0mm) {$1$};
\draw[black, thick] (-50mm,0mm) circle (4.1mm);

\node[] at (0mm,0mm) {$2$};
\draw[black, thick] (0mm,0mm) circle (4.1mm);

\node[] at (50mm,0mm) {$3$};
\draw[black, thick] (50mm,0mm) circle (4.1mm);

\draw[-{Latex[length=2mm, width=2mm]}, black, very thick] (Ltr) to[out=45,in=135] (Mtl);
\node[] at (-25mm,7mm) {$A_{1, 2}^{(1)}$};

\draw[-{Latex[length=2mm, width=2mm]}, black, very thick] (Ltl) to[out=65,in=115] (Rtr);
\node[] at (0mm,26mm) {$A_{1, 3}^{(1)}$};

\draw[-{Latex[length=2mm, width=2mm]}, black, very thick] (Mtr) to[out=45,in=135] (Rtl);
\node[] at (25mm,7mm) {$A_{2, 3}^{(1)}$};

\draw[-{Latex[length=2mm, width=2mm]}, black, very thick] (Mbr) to[out=-45,in=-135] (Rbl);
\node[] at (25mm,-7mm) {$A_{2, 3}^{(-1)}$};

\draw[-{Latex[length=2mm, width=2mm]}, black, very thick] (Lbr) to[out=-45,in=-135] (Mbl);
\node[] at (-25mm,-7mm) {$A_{1, 2}^{(-1)}$};

\draw[-{Latex[length=2mm, width=2mm]}, black, very thick] (Lbl) to[out=-65,in=-115] (Rbr);
\node[] at (0mm,-26mm) {$A_{1, 3}^{(-1)}$};

\end{tikzpicture}
\caption{The arrows represented by members $A_{i, j}^{(k)}$ of $\genset{3}$ with $i < j$. The objects in our groupoid $\{1, 2, 3 \}$ are represented by circles. Each arrow points from $\source(A_{i, j}^{(k)})$ to $\target(A_{i, j}^{(k)})$.}
\label{fig:letters_schematic}
\end{figure}
An important consequence of the properties of our groupoid is that each non-unit arrow can be uniquely represented as a product of elements of $\genset{N}$ which alternate between $\pm 1$ in the upper index $(k)$. More precisely, we have the following lemma which is proved in Appendix~\ref{sec:appendix_proofs}.
\begin{lemma} \label{lemma:plane_switching}
Each arrow $w \in \arrowset$ can be represented as a reduced word in a unique way. This reduced word is either an empty word or, for some $d \geq 1$, a product of the form
\begin{equation} \label{eq:generic_reduced_word}
    w = A_{i_1, i_2}^{(k)} A_{i_2, i_3}^{(-k)} A_{i_3, i_4}^{(k)} \cdots  A_{i_d, i_{d + 1}}^{((-1)^{d + 1} k)} \text{ such that } i_\ell \neq i_{\ell + 1} \text{ for all } 1 \leq \ell \leq d.
\end{equation}
\end{lemma}
We say that $|\cdot|_{\metric}:\arrowset\to [0,\infty)$ is a metric on $\arrowset$ generated by $\genset{N}$ if $|e_i|_\metric = 0$ for all $i$ and for any nonempty $w\in \arrowset$ we have
\begin{equation} \label{eq:metric_condition}
    |w|_{\metric}  = \sum_{\ell = 1}^{d} \left|A_{i_\ell , j_\ell}^{(k_\ell)}\right|_{\metric},
\end{equation}
where $ \prod_{\ell = 1}^{d} A_{i_\ell , j_\ell}^{(k_\ell)}$ is the unique reduced representation of $w$ given by Lemma~\ref{lemma:plane_switching}. We reserve $|\cdot|$ to denote the number of generators in the reduced representation of  $w\in \arrowset$;
\begin{equation} \label{eq:word_length_metric_condition}
    |w|= \left|\prod_{\ell = 1}^{d} A_{i_\ell , j_\ell}^{(k_\ell)}\right| = d.
\end{equation}

\subsection{The Markov chain on \texorpdfstring{$\groupoid_N$}{G(N)}}
\label{sec:markov_chain}

We study discrete time Markov chains $\{W_{n}\}_{n \geq 0}$ on the arrow set $\arrowset$. We assume the following:
\begin{assumption}\label{ass:append_A_each_step}
In each step the value of $W_n$ changes by the right composition of a generating element $A_{i, j}^{(k)}$, i.e. $W_{n}^{-1} W_{n + 1} \in \genset{N}$ for all $n \geq 0$.
\end{assumption}
\begin{assumption} \label{ass:prob_form}
The conditional probability of the increments given the starting state depends on the starting state only through its target.
\end{assumption}
These two assumptions imply that for $x, y \in \arrowset$ the transition probability function is of the form
\begin{equation} \label{eq:full_transition_probs}
P\left( W_{n + 1} = y \st W_{n} = x \right) \defn
\begin{cases}
p_{i, j}^{(k)} & \text{if } x^{-1} y = A_{i , j}^{(k)} \text{ and } i \neq j  \\
0 & \text{otherwise}.
\end{cases}
\end{equation}
We further assume:
\begin{assumption} \label{ass:probs_positive}
The jump probabilities satisfy $p_{i, j}^{(k)} \in (0, 1) $ for each $(i, j, k)$.
\end{assumption}

We consider $N \geq 3$ since, if $N = 2$, then the process $\{W_{2n}\}_{n \geq 0}$ reduces to the Abelian case, namely a Markov chain on the free group of rank 1, namely a lazy random walk on $\mathbb{Z}$.

From Eq.~\eqref{eq:full_transition_probs}, we see that there are $2N(N - 1)$ transition probabilities and we must have $\sum_{j, k} p_{i, j}^{(k)} = 1$ for each $1 \leq i \leq N$. We call $W_n$ the ``word'' at time $n$. We use the notation $P_x(\cdot)$ and $E_x[\cdot]$ to indicate that the probabilities and expectations in question are calculated under the initial condition $W_0 = x$.

Define the sets
\begin{equation} \label{eq:irreducible_arrowset}
    \arrowset_i \defn \{w: \source(w) = i, w \in \arrowset \}, \quad 1 \leq i \leq N.
\end{equation}
Some important consequences of our assumptions are collected in the following lemma which is proved in Appendix~\ref{sec:appendix_proofs}.
\begin{lemma} \label{lemma:intermediate_words}
\hspace{20cm} 
\begin{enumerate}[label=(\roman*), leftmargin=* ]
    \item For any $x \in \arrowset$, under $P_x$ the process $\{x^{-1} W_n\}_{n\ge 0}$ has the same distribution as the process $\{ W_n\}_{n\ge 0}$ under $P_{e_{\target(x)}}$. \label{lemma:intermediate_words_markov}
    \item There is a positive probability path in the Markov chain $\{W_n\}_{n \geq 0}$ between two words $w_1 \neq w_2$ if and only if $w_1, w_2 \in \arrowset_i$ for some $1 \leq i \leq N$. \label{lemma:intermediate_words_sources}
    \item Suppose that $\{v_n\}_{1 \leq n \leq n'}$ is such a path where $v_1 = w_1$ and $v_{n'} = w_2$. Consider the reduced composition
\begin{equation}
    w_1^{-1} w_2= \prod_{\ell = 1}^{d} A_{i_\ell , j_\ell}^{(k_\ell)} \quad d \geq 1
\end{equation}
given by Lemma~\ref{lemma:plane_switching}. Then, there exist times $\{n_m\}_{1 \leq m\leq d - 1}$ satisfying $1 <  n_1 < n_2< \dots <n_{d - 1} < n'$, such that $v_{n_m} = w_1 \prod_{\ell = 1}^{m} A_{i_\ell , j_\ell}^{(k_\ell)}$ for each $1 \leq m \leq d- 1$. \label{lemma:intermediate_words_intermediate}
\end{enumerate}
\end{lemma}

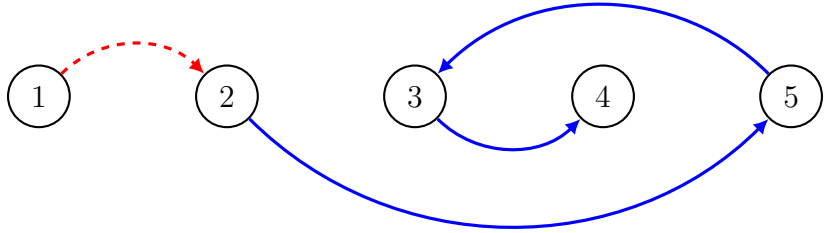
\begin{figure}[ht]
\centering
\begin{tikzpicture}

\coordinate (LLtr) at (-47mm,3mm);
\coordinate (LLtl) at (-53mm,3mm);

\coordinate (LLbr) at (-47mm,-3mm);
\coordinate (LLbl) at (-53mm,-3mm);

\coordinate (Ltr) at (-22mm,3mm);
\coordinate (Ltl) at (-28mm,3mm);

\coordinate (Lbr) at (-22mm,-3mm);
\coordinate (Lbl) at (-28mm,-3mm);

\coordinate (Mtl) at(-3mm,3mm);
\coordinate (Mtr) at(3mm,3mm);

\coordinate (Mbl) at(-3mm,-3mm);
\coordinate (Mbr) at(3mm,-3mm);

\coordinate (Rtr) at (28mm,3mm);
\coordinate (Rtl) at (22mm,3mm);

\coordinate (Rbr) at (28mm,-3mm);
\coordinate (Rbl) at (22mm,-3mm);

\coordinate (RRtl) at (47mm,3mm);
\coordinate (RRtr) at (53mm,3mm);

\coordinate (RRbl) at (47mm,-3mm);
\coordinate (RRbr) at (53mm,-3mm);

\node[] at (-50mm,0mm) {$1$};
\draw[black, thick] (-50mm,0mm) circle (4.1mm);

\node[] at (-25mm,0mm) {$2$};
\draw[black, thick] (-25mm,0mm) circle (4.1mm);

\node[] at (0mm,0mm) {$3$};
\draw[black, thick] (0mm,0mm) circle (4.1mm);

\node[] at (25mm,0mm) {$4$};
\draw[black, thick] (25mm,0mm) circle (4.1mm);

\node[] at (50mm,0mm) {$5$};
\draw[black, thick] (50mm,0mm) circle (4.1mm);

\draw[-{Latex[length=2mm, width=2mm]}, red, very thick, dashed] (LLtr) to[out=45,in=135] (Ltl);

\draw[-{Latex[length=2mm, width=2mm]}, blue, very thick] (Lbr) to[out=-45,in=-135] (RRbl);

\draw[-{Latex[length=2mm, width=2mm]}, blue, very thick] (RRtl) to[out=135,in=45] (Mtr);

\draw[-{Latex[length=2mm, width=2mm]}, blue, very thick] (Mbr) to[out=-45,in=-135] (Rbl);

\end{tikzpicture}
\caption{An illustration for Statement~\ref{lemma:intermediate_words_intermediate} of Lemma~\ref{lemma:intermediate_words}.
Let  $w_1 = A_{1 , 2}^{(1)}$ (red, dashed) and $w_1^{-1} w_2 = A_{2 , 5}^{(-1)} A_{5 , 3}^{(1)} A_{3 , 4}^{(-1)}$ (blue, solid) in the $N = 5$ case. A path joining $w_1$ and $w_2$ must visit the intermediate words $w_1 A_{2, 5}^{(-1)}$, $w_1 A_{2, 5}^{(-1)} A_{5, 3}^{(1)}$ in this order.}
\label{fig:multiple_R_path_example}
\end{figure}

\begin{remark} \label{remark:physical_situation_connection}
To connect with the original problem described in Section~\ref{sec:motivation}, we first note that any path in the domain $D$  (described in Figure ~\ref{fig:path}) starting at a window and ending at a window can be naturally mapped into an element of $\arrowset$. Indeed, an element of $\arrowset$ corresponds to the homotopy class of a path connecting two windows, where we allow the starting and ending points of the path to move inside a window.

Let $B_t$ be the position of the reflected Brownian motion in $D$ at time $t$. Let $\tau_0 = \inf_{s \geq 0}\{B_s \text{ at a window}\}$ and
\begin{equation}
\tau_{n + 1} = \inf_{s \geq \tau_n}\{ B_{s} \text{ at a window different from the window visited at time } \tau_n \}.
\end{equation}
Denoting by $W_n$ the element in $\arrowset$ corresponding to the Brownian path in the time interval $[\tau_0, \tau_n]$, we see that the sequence $(W_n, B_{\tau_n})_{n\ge 0}$ is a Markov chain on $\arrowset\times D$. By the narrow escape problem  \cite{HolcmanSchuss} the transition distribution of this Markov chain will only depend on the first coordinate (an element in $\arrowset$) as all window sizes go to 0, which leads to a Markov chain on $\arrowset$ satisfying the conditions described  in Section \ref{sec:markov_chain}. The narrow escape problem also implies that understanding the growth of the length of $W_n$ in $n$ allows us to understand the growth of the entanglement of the Brownian path in $t$.
\end{remark}

\section{Statement of main result}
\label{sec:main}

We now state our main result, which is the computation of the almost-sure limit $\lim_{n \to \infty}|W_n|_{\metric}/n$ as well as a central limit theorem. To do this, we require a set of functions whose properties are collected in the following proposition.

\begin{proposition} \label{prop:R_properties}
    There is an $\epsilon > 0$ and a unique set of complex functions
    \begin{equation}
    \{R_{i, j}^{(k)}:D_{1+\epsilon}\to \isset{C} : i \neq j , \ \ 1 \leq i,  j \leq N, \ \ k \in \{-1, 1\} \},\qquad D_r=\{z:|z|< r\},
\end{equation}
  satisfying the following properties:
\begin{enumerate}[label=(R\arabic{*})]
    \item \label{eq:R_property_radius} Each $R_{i, j}^{(k)}$ is complex analytic in the disk $D_{1+\epsilon}$.

     \item \label{eq:R_property_physical} If $\lambda \in [0, 1]$, then $R_{i, j}^{(k)}(\lambda) \in (0, 1)$.

     \item \label{eq:R_property_equations} They satisfy the system of equations
     \begin{equation} \label{eq:fund_R_eqs}
    R_{i, j}^{(k)} = \lambda \left[ p_{i, j}^{(k)} + \sum_{m \neq i, j } p_{i, m}^{(k)}  R_{m, j}^{(k)} + \sum_{m \neq i } p_{i, m}^{(-k)} R_{m, i}^{(-k)} R_{i, j}^{(k)}   \right], \quad \text{for }\lambda\in D_{1+\epsilon}.
    \end{equation}
 \end{enumerate}
\end{proposition}
Let $\ismat{B}:\{-1,1\}\times D_{1+\epsilon}\times \isset{C}\to \isset{C}^{N\times N}$ be a matrix-valued function whose $i, j^{\text{th}}$ entry is
\begin{equation} \label{eq:B_matrix_intro}
\matinds{B}{i}{j}(k; \lambda, z) \defn (1 - \delta_{i, j})\, z^{|A_{i , j}^{(k)} |_{\metric} } R_{i, j}^{(k)}(\lambda)
\end{equation}
where $\delta_{i, j}$ is the Kronecker delta.  Now let
\begin{equation} \label{eq:g_intro}
    h(\lambda, z) \defn  \det\left[\idmat - \ismat{B}(1; \lambda, z)\, \ismat{B}(-1; \lambda, z) \right],
\end{equation}
where $\idmat$ is the identity matrix. Let
\begin{subequations}
\label{eq:gamma_sigma}
\begin{align}
    \gamma &\defn \frac{\partial_z h(1, 1)}{\partial_\lambda h(1, 1)}, \label{eq:gamma_intro} \\
    \sigma^2 &\defn \frac{\partial_z^2 h(1, 1) + \partial_z h(1, 1) - 2 \gamma \, \partial_{z, \lambda}^2 h(1, 1) + \gamma^2 \, \left(\partial_\lambda^2 h(1, 1) + \partial_\lambda h(1, 1) \right) }{\partial_\lambda h(1, 1)}. \label{eq:sigma_intro}
\end{align}
\end{subequations}
We are now ready to state our main result:
\begin{theorem} \label{thm:main}
Consider the Markov chain $\{W_n\}_{n \geq 0}$ satisfying Assumptions~\ref{ass:append_A_each_step}, \ref{ass:prob_form} and \ref{ass:probs_positive} and having transition probabilities defined in Eq.~\eqref{eq:full_transition_probs}. Then, for any initial condition $W_0$ we have
\begin{subequations}
\begin{align}
    \lim_{n \to \infty} \frac{|W_{n}|_\metric}{n} &= \gamma \quad a.s. \label{eq:main_lln}\\
    \frac{|W_{n}|_\metric - \gamma n}{\sqrt{n}} &\to \mathcal{N}(0, \sigma^2) \quad \text{in law}. \label{eq:main_clt}
\end{align}
\end{subequations}
The constants $\gamma$ and $\sigma^2$ are defined in Eq.~\eqref{eq:gamma_sigma} and $\mathcal{N}(0, \sigma^2)$ denotes the normal distribution with mean $0$ and variance $\sigma^2$.
\end{theorem}

\begin{remark}
\label{remark:gilch}
Gilch \cite{gilch2008} studies similar problems in the context of random walks on regular languages. A random walk on a regular language as defined in \cite{Lalley2001} is a Markov chain on the set of all finite words from a finite alphabet with the following conditions. In one jump only the last two letters of a word may be modified and at most one letter may be adjoined or deleted. The transition probabilities only depend on the last two letters of the current word. Our process $\{W_{n}\}_{n \geq 0}$ is a random walk on a regular language formed by the alphabet $\genset{N}$.

Theorem~2.4 in \cite{gilch2008} provides a law of large numbers under the assumption that the considered random walk is transient. This could potentially lead to another way to obtain the constant $\gamma$ in Theorem~\ref{thm:main}. However, the identification of the constant appearing in Theorem~2.4 of \cite{gilch2008} requires the solution of a more complicated problem than in our case.
\end{remark}

\section{Examples}
\label{sec:examples}

Here we demonstrate several applications of Theorem~\ref{thm:main}. We will consider two simple metrics and compute the constants in Eqs.~\eqref{eq:gamma_intro} and \eqref{eq:sigma_intro}. The first is the metric $|\cdot|$ defined in Eq.~\eqref{eq:word_length_metric_condition}; the second is the metric $| \cdot |_{\metrictwo}$ generated by $\genset{N}$ defined for the generators as
\begin{equation}
    |A_{i, j}^{(k)}|_{\metrictwo} = |i - j|.
\end{equation}

\subsection{The one parameter case with \texorpdfstring{$N = 3$}{N = 3}}

We take $N = 3$ and the set of transition probabilities
\begin{align*}
p_{2, 1}^{(1)} = p_{2, 3}^{(1)} = p_{2, 1}^{(-1)} = p_{2, 3}^{(-1)} &= 1/4, \\
p_{1, 2}^{(1)} = p_{3, 2}^{(1)} = p_{1, 2}^{(-1)} = p_{3, 2}^{(-1)} &= q, \\
p_{1, 3}^{(1)} = p_{3, 1}^{(1)} = p_{1, 3}^{(-1)} = p_{3, 1}^{(-1)} &= 1/2 - q ,
\end{align*}
for $0 < q < 1/2$.  This represents a situation where the planar domain of Section~\ref{sec:motivation} is left-right and up-down symmetric about its center.  Let $\gamma_3(q), \sigma_3^2(q)$ be the constants appearing in Theorem~\ref{thm:main} for the $|\cdot|$ metric and let $\gamma_{3, \metrictwo}(q), \sigma_{3, \metrictwo}^2(q)$ be the same constants for the $|\cdot|_{\metrictwo}$ metric. We will show that
\begin{subequations}
\begin{align}
\gamma_3(q) &= \frac{3q + (1 - 4q)\sqrt{(8 - 7q)q}}{4(1 - 4q^2)},  \label{eq:3_window_natural_metric_gamma} \\
\gamma_{3, \metrictwo}(q) &= \frac{\sqrt{(8 - 7q)q} - q}{2(2q + 1)},  \label{eq:3_window_F_metric_gamma} \\
    \sigma^2_{3}(q) &= \tfrac{4 (8 + 5 Q) + (68 - 56 Q) q + (500 - 101 Q) q^2 - (1471 +
    64 Q) q^3 + 8 (1 + 42 Q) q^4 + 728 q^5}{8 (1 + 2 q)^3 (8 - 23 q + 14 q^2)}, \label{eq:3_window_natural_metric_sigma} \\
    \sigma^2_{3, \metrictwo}(q) &= \tfrac{(32 + 4 Q) + (36 + 28 Q) q + (80 - 21 Q) q^2 - (199 + 30 Q) q^3 +
 70 q^4}{2 (1 + 2 q)^3 Q^2}, \label{eq:3_window_F_metric_sigma}
\end{align}
\end{subequations}
where $Q = \sqrt{(8 - 7q)q}$. The constants $\gamma_3(q), \gamma_{3, \metrictwo}(q)$ are plotted in Figure~\ref{fig:3_window_gamma_plot} and the constants $\sigma^2_{3}(q), \sigma^2_{3, \metrictwo}(q)$ are plotted in Figure~\ref{fig:3_window_sigma_plot}.

By our choice of probabilities, Eq.~\eqref{eq:fund_R_eqs} reduces to a set of three equations by symmetry. We define
\begin{align*}
R_1 &\defn R_{2, 1}^{(1)} = R_{2, 3}^{(1)} = R_{2, 1}^{(-1)} = R_{2, 3}^{(-1)},  \\
R_2 &\defn R_{1, 2}^{(1)} = R_{3, 2}^{(1)} = R_{1, 2}^{(-1)} = R_{3, 2}^{(-1)},  \\
R_3 &\defn R_{1, 3}^{(1)} = R_{3, 1}^{(1)} = R_{1, 3}^{(-1)} = R_{3, 1}^{(-1)},
\end{align*}
which leads to
\begin{subequations}
\label{eq:3_window_r_eqs}
\begin{align}
R_1 &= \tfrac14\lambda\left[1 + R_3 + 2 R_2 R_1 \right], \\
R_2 &= \lambda\left[q + \left(\tfrac{1}{2} - q \right)R_2 + q R_1 R_2 + \left(\tfrac{1}{2} - q \right)R_3 R_2  \right],\\
R_3 &= \lambda \left[\tfrac{1}{2} - q + q R_1 + q R_1 R_3 + \left(\tfrac{1}{2} - q \right)R_3^2 \right].
\end{align}
\end{subequations}
Substituting $\lambda = 1$ gives a cubic equation for $R_2(1)$. We choose the solution which satisfies $R_2(1) \in (0, 1)$ (by property~\ref{eq:R_property_physical}). By implicit differentiation of Eq.~\eqref{eq:3_window_r_eqs} it follows that
\begin{subequations}
\label{eq:one_parameter_R_vals}
\begin{alignat}{3}
R_1(1) &= \tfrac{1}{2};  &&R_1'(1) &&= \frac{3q + 2 + Q}{2(Q - q)} \label{eq:one_parameter_R_vals_1} \\
R_2(1) &= \frac{3q - Q}{2(2q - 1)}; &&R_2'(1) &&= \frac{2(q + 1)}{Q} \label{eq:one_parameter_R_vals_2} \\
R_3(1) &= \frac{q - 2 + Q}{2(2q - 1)};  \quad && R_3'(1) &&= \frac{2Q + q(Q - 6 + q(5 - 4q - 4Q))}{q(2q - 1)(Q + 7q - 8)} \label{eq:one_parameter_R_vals_3}
\end{alignat}
\begin{align}
R_1''(1) &= \tfrac{4 Q + 16 (2 + Q) q +  2 (42 + 19 Q) q^2  (26 + 31 Q) q^3 -
 3 (53 + 4 Q) q^4 + 84 q^5}{4 (1 - q)^2 q^2 Q^2}, \\
 R_2''(1) &= \tfrac{8 (3 + Q) + (25 Q - 12) q +  (49 + 4 Q) q^2 - 4 (9 + 7 Q) q^3 -
 16 q^4}{(1 - q) Q^3}, \\
 R_3''(1) &= \tfrac{-16 (1 + Q) + (48 - 34 Q) q + 6 (23 Q -43 ) q^2 +
 2 (329 + 18 Q) q^3 - (176 + 137 Q) q^4 + (28 Q -541) q^5 + 300 q^6}{2  (1 - q)^2  (2 q - 1) Q^3}.
\end{align}
\end{subequations}
Since our $p_{i, j}^{(k)}$ are symmetric with respect to $k = \pm 1$, in the $|\cdot| $ metric, we construct
\begin{equation} \label{eq:Bmat_3_window_K}
\ismat{B}(1; \lambda, z) = \ismat{B}(-1;\lambda, z) = \begin{pmatrix}
0 &  z R_2 & z R_3 \\
z R_1 &  0 & z R_1 \\
z R_3 & z R_2 & 0
\end{pmatrix},
\end{equation}
so that
\begin{equation}
h(\lambda, z) = (1 - z^2 R_3^2) \left[(1 - 2 z^2 R_1 R_2)^2 - z^2 R_3^2 \right],
\end{equation}
Substituting this into Eqs.~\eqref{eq:gamma_intro} and \eqref{eq:sigma_intro} and using Eq.~\eqref{eq:one_parameter_R_vals} gives the required constants. Similarly, in the $|\cdot|_{\metrictwo}$ metric, we have
\begin{equation} \label{eq:Bmat_3_window_W}
\ismat{B}(1; \lambda, z) = \ismat{B}(-1;\lambda, z)  = \begin{pmatrix}
0 &  z R_2 & z^2 R_3 \\
z R_1 &  0 & z R_1 \\
z^2 R_3 & z R_2 & 0
\end{pmatrix}.
\end{equation}
Hence, we find
\begin{equation}
h(\lambda, z) = (1 - z^4 R_3^2) \left[(1 - 2 z^2 R_1 R_2)^2 - z^4 R_3^2 \right].
\end{equation}
We again substitute this into Eqs.~\eqref{eq:gamma_intro} and \eqref{eq:sigma_intro} and use Eq.~\eqref{eq:one_parameter_R_vals} to obtain the required constants.
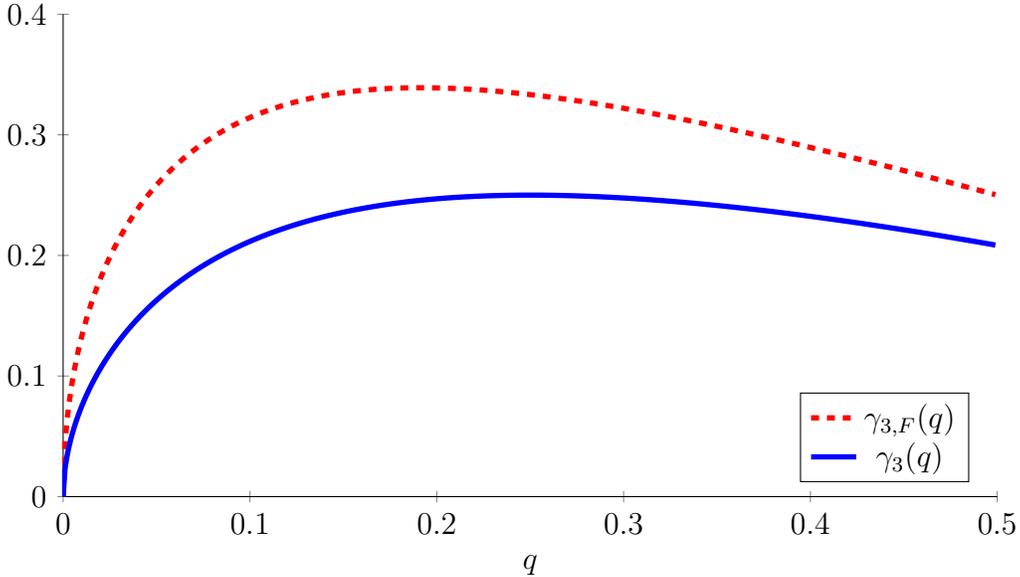
\begin{figure}[ht]
    \centering
\begin{tikzpicture}
\begin{axis}[
    height = 8cm,
    width = 14cm,
    xmin = 0,
    xmax = 0.5,
    ymin = 0,
    ymax = 0.4,
    axis lines = left,
    xlabel = $q$,
    xtick = {0, 0.1, 0.2, 0.3, 0.4, 0.5},
    ytick = {0, 0.1, 0.2, 0.3, 0.4},
    legend pos = south east,
    axis line style={-}
]

\addplot[line width = 2pt, dashed, color = red]
table[col sep = comma]{plotdata/gamma3Fdata.csv};
\addlegendentry{$\gamma_{3, \metrictwo}(q)$ }

\addplot[line width = 2pt, solid, color = blue]
table[col sep = comma]{plotdata/gamma3data.csv};
\addlegendentry{$\gamma_3(q)$}

\end{axis}
\end{tikzpicture}
\caption{The constants $\gamma_3(q)$ of Eq. \eqref{eq:3_window_natural_metric_gamma}  and $\gamma_{3, \metrictwo}(q)$ of Eq. \eqref{eq:3_window_F_metric_gamma}. We note that $\gamma_3(q)$ has a maximum value of $1/4$ when $q = 1/4$ and $\gamma_{3, \metrictwo}(q)$ has a maximum value of $(2/23)(2 \sqrt{6} - 1)$ when $q = (8 - \sqrt{6})/29$.}
\label{fig:3_window_gamma_plot}
\end{figure}

\begin{figure}
     \centering
     \begin{subfigure}[t]{0.45\linewidth}
        \centering
        \begin{tikzpicture}
        \begin{axis}[
            height = 5cm,
            width = 8cm,
            xmin = 0,
            xmax = 0.5,
            ymin = 0.5,
            ymax = 0.7,
            axis lines = left,
            xlabel = $q$,
            xtick = {0, 0.1, 0.2, 0.3, 0.4, 0.5},
            ytick = {0.5, 0.6, 0.7},
            axis line style={-}
        ]

        \addplot[line width = 2pt, solid, color = black]
        table[col sep = comma]{plotdata/variance3q.csv};

        \end{axis}
        \end{tikzpicture}

        \caption{}
        \label{fig:3_window_sigma_plot_nat}
     \end{subfigure}
     \begin{subfigure}[t]{0.45\linewidth}
        \centering
        \begin{tikzpicture}
        \begin{axis}[
            height = 5cm,
            width = 8cm,
            xmin = 0,
            xmax = 0.5,
            ymin = 0.7,
            ymax = 2.5,
            axis lines = left,
            xlabel = $q$,
            xtick = {0, 0.1, 0.2, 0.3, 0.4, 0.5},
            ytick = {1.0, 1.5, 2.0, 2.5},
            axis line style={-}
        ]

        \addplot[line width = 2pt, solid, color = black]
        table[col sep = comma]{plotdata/variance3qF.csv};

        \coordinate (insetPosition) at (axis cs:0.5,2.5);
        \end{axis}

        \begin{axis}[
            at={(insetPosition)},
            anchor={outer north east},
            height = 3cm,
            width = 4cm,
            ymax = 2.018,
            ytick = {2, 2.006, 2.012, 2.018},
            xlabel = $q \times 10^3$,
            x label style={font = \tiny, at={(axis description cs:0.5,-0.25)},anchor=north},
            xticklabel style={font = \tiny},
            yticklabel style={font = \tiny, /pgf/number format/.cd,fixed zerofill,precision=3},
            axis lines = left,
            axis line style={-}
        ]

        \addplot[line width = 1.5pt, solid, color = black]
        table[col sep = comma]{plotdata/variance3qFMIN.csv};

        \end{axis}
        \end{tikzpicture}

        \caption{}
        \label{fig:3_window_sigma_plot_F}
     \end{subfigure}

    \caption{(A) The quantity $\sigma_{3}^2(q)$ of Eq.~\eqref{eq:3_window_natural_metric_sigma}. We note that $\sigma_{3}^2(q)$ has a maximum value of $11/16$ when $q = 1/4$. (B) The quantity $\sigma_{3, \metrictwo}^2(q)$ of Eq.~\eqref{eq:3_window_F_metric_sigma}. We note that $\sigma_{3, \metrictwo}^2(q)$ has a very slight maximum value of $2.01584\dots$ when $q = 0.00205319\dots$ (inset). The exact value of $q$ for which $\sigma_{3, \metrictwo}^2(q)$ is maximized is a root of a certain eighth-order polynomial.}
    \label{fig:3_window_sigma_plot}
\end{figure}
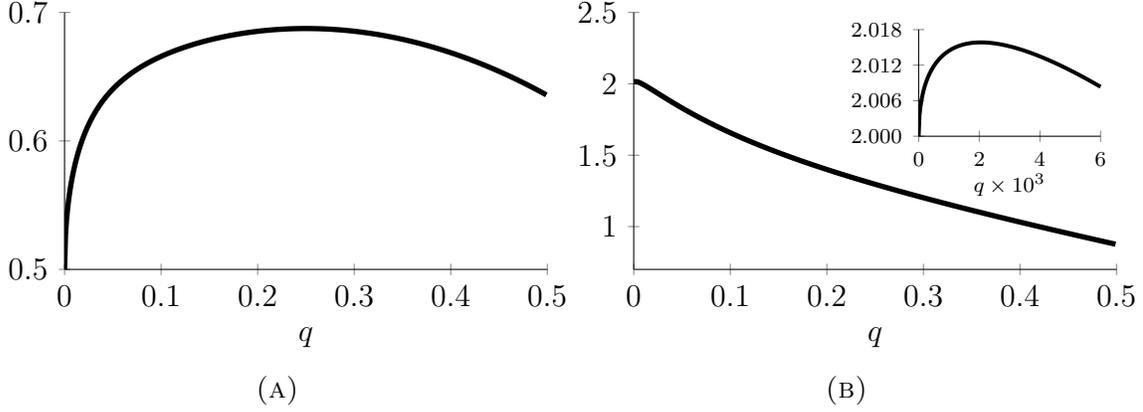
\subsection{An asymmetric case with \texorpdfstring{$N = 3$}{N = 3}}

We take $N = 3$ and the following arbitrarily chosen set of probabilities:
\begin{alignat*}{7}
p_{2, 1}^{(1)} &= 17/40; \quad &&p_{2, 3}^{(1)} &&= 1/5; \quad  &&p_{2, 1}^{(-1)} &&= 1/8; \quad &&p_{2, 3}^{(-1)} &&= 1/4 \\
p_{1, 2}^{(1)} &= 43/70; \quad &&p_{3, 2}^{(1)} &&= 43/72; \quad &&p_{1, 2}^{(-1)} &&= 1/7; \quad &&p_{3, 2}^{(-1)} &&= 1/8 \\
p_{1, 3}^{(1)} &= 1/10; \quad &&p_{3, 1}^{(1)} &&= 1/9; \quad &&p_{1, 3}^{(-1)} &&= 1/7; \quad &&p_{3, 1}^{(-1)} &&= 1/6.
\end{alignat*}
Truncating to six significant digits, solving Eq.~\eqref{eq:fund_R_eqs} numerically gives, for $r_{i, j}^{(k)} \defn R_{i, j}^{(k)}(1)$,
\begin{alignat*}{7}
r_{2, 1}^{(1)} &= 0.591572; \quad &&r_{2, 3}^{(1)} &&= 0.404666; \quad  &&r_{2, 1}^{(-1)} &&= 0.388890; \quad &&r_{2, 3}^{(-1)} &&= 0.579542 \\
r_{1, 2}^{(1)} &= 0.769190; \quad &&r_{3, 2}^{(1)} &&= 0.791039; \quad &&r_{1, 2}^{(-1)} &&= 0.305398; \quad &&r_{3, 2}^{(-1)} &&= 0.245890 \\
r_{1, 3}^{(1)} &= 0.386687; \quad &&r_{3, 1}^{(1)} &&= 0.538119; \quad &&r_{1, 3}^{(-1)} &&= 0.387184; \quad &&r_{3, 1}^{(-1)} &&= 0.300936,
\end{alignat*}
and, for $d_{i, j}^{(k)} \defn \text{d} R_{i, j}^{(k)}(\lambda)/\text{d}\lambda |_{\lambda = 1}$,
\begin{alignat*}{7}
d_{2, 1}^{(1)} &= 1.36978; \quad &&d_{2, 3}^{(1)} &&= 1.37284; \quad  &&d_{2, 1}^{(-1)} &&= 2.05937; \quad &&d_{2, 3}^{(-1)} &&= 2.69097 \\
d_{1, 2}^{(1)} &= 1.44102; \quad &&d_{3, 2}^{(1)} &&= 1.71828; \quad &&d_{1, 2}^{(-1)} &&= 1.31219; \quad &&d_{3, 2}^{(-1)} &&= 0.991008 \\
d_{1, 3}^{(1)} &= 1.56411; \quad &&d_{3, 1}^{(1)} &&= 1.99059; \quad &&d_{1, 3}^{(-1)} &&= 2.01524; \quad &&d_{3, 1}^{(-1)} &&= 1.19855,
\end{alignat*}
and, for $v_{i, j}^{(k)} \defn \text{d}^2 R_{i, j}^{(k)}(\lambda)/\text{d}\lambda^2 |_{\lambda = 1}$,
\begin{alignat*}{7}
v_{2, 1}^{(1)} &= 10.3365; \quad &&v_{2, 3}^{(1)} &&= 12.8916; \quad  &&v_{2, 1}^{(-1)} &&= 26.2278; \quad &&v_{2, 3}^{(-1)} &&= 32.1490 \\
v_{1, 2}^{(1)} &= 9.45100; \quad &&v_{3, 2}^{(1)} &&= 13.7182; \quad &&v_{1, 2}^{(-1)} &&= 15.3088; \quad &&v_{3, 2}^{(-1)} &&= 11.6415 \\
v_{1, 3}^{(1)} &= 15.5010; \quad &&v_{3, 1}^{(1)} &&= 18.8416; \quad &&v_{1, 3}^{(-1)} &&= 26.1337; \quad &&v_{3, 1}^{(-1)} &&= 14.3857.
\end{alignat*}

Here, $\ismat{B}(1; \lambda, z)$ and $\ismat{B}(- 1; \lambda, z)$ will be different; however they are each $3 \times 3$ matrices with the same structure as Eqs.~\eqref{eq:Bmat_3_window_K} and \eqref{eq:Bmat_3_window_W} in the corresponding metric. Via Eqs.~\eqref{eq:gamma_intro} and \eqref{eq:sigma_intro} we find the constants $\gamma_{\text{a}}, \sigma^2_{\text{a}}$ in the $|\cdot|$ metric and $\gamma_{\text{a}, \metrictwo}, \sigma^2_{\text{a}, \metrictwo}$ in the $|\cdot|_{\metrictwo}$ metric to be
\begin{subequations}
\begin{alignat}{2}
\gamma_{\text{a}} &= 0.272913\dots  \quad \sigma^2_{\text{a}} &&= 0.587598\dots \\
\gamma_{\text{a}, \metrictwo} &= 0.334211\dots \quad \sigma^2_{\text{a}, \metrictwo} &&= 0.916276\dots
\end{alignat}
\end{subequations}
\subsection{The totally symmetric case}

For any $N \geq 3$, we take $p_{i, j}^{(k)} = 1/(2N - 2)$ for all $(i, j, k)$. Let $\gamma_\text{sym}, \sigma_\text{sym}^2$ be the constants appearing in Theorem~\ref{thm:main} for the $|\cdot|$ metric and let $\gamma_{\text{sym}, \metrictwo}, \sigma_{\text{sym}, \metrictwo}^2$ be the same constants for the $|\cdot|_{\metrictwo}$ metric. We will show that
\begin{subequations}
\begin{align}
    \gamma_{\text{sym}} &= \frac{N - 2}{2(N - 1)}, \label{eq:symmetric_natural_metric_gamma} \\
    \gamma_{\text{sym}, \metrictwo} &=  \frac{(N + 1)(N - 2)}{6(N - 1)} , \label{eq:symmetric_F_metric_gamma} \\
    \sigma^2_{\text{sym}} &= \frac{N^2 + 2N - 4}{4(N - 1)^2}, \label{eq:symmetric_natural_metric_sigma} \\
  \sigma^2_{\text{sym}, \metrictwo} &= \frac{11 N^5 - 2 N^4 + 15 N^3 - 36 N - 8}{180 N (N - 1)^2}. \label{eq:symmetric_F_metric_sigma}
\end{align}
\end{subequations}
Note that Eqs.~\eqref{eq:symmetric_natural_metric_gamma} and \eqref{eq:symmetric_F_metric_gamma} also hold for $N = 2$ since $\{W_n\}_{n \geq 0}$ is recurrent in this case. We will use the following lemma for the characteristic polynomial of a Kac--Murdock--Szeg{\H{o}} matrix \cite{kms1953}.
\begin{lemma} \label{lemma:det_kms}
Let $\ismat{U}_n(z)$ be an $n \times n$ matrix whose $i, j^{\text{th}}$ entry is $[\ismat{U}_{n}(z)]_{i, j}= z^{|i-j|}$. Let $\phi_n(x, z) = \textup{det}\left[\ismat{U}_n(z) - x \idmat\right]$. Then, defining $\phi_0(x, z) = 1$, we have
\begin{equation} \label{eq:kms_result}
    \phi_{n}(x, z) = (1 - x - z^2(1 + x)) \phi_{n - 1}(x, z) - x^2 z^2 \phi_{n - 2}(x, z); \phi_0(x, z) = 1, \phi_1(x, z) = 1 - x.
\end{equation}
\end{lemma}
\begin{proof}[Proof of Lemma~\ref{lemma:det_kms}]
We multiply the second row of $\ismat{U}_n(z) - x\idmat$ by $-z$ and add it to the first row. Then, we multiply the second column of the resulting matrix by $-z$ and add it to the first column. The result is that
\begin{equation}
\phi_{n}(x, z) = \begin{vmatrix}
1 - x & z & z^2 & \dots \\
z & 1- x & z & \dots \\
z^2 & z & 1 - x & \dots \\
\vdots & \vdots & \vdots & \ddots
\end{vmatrix} = \begin{vmatrix}
1 - x - z^2(1 + x) & x z & 0 & \dots \\
x z & 1- x & z & \dots \\
0 & z & 1 - x & \dots \\
\vdots & \vdots & \vdots & \ddots
\end{vmatrix} .
\end{equation}
Expanding the determinant along the first column gives
\begin{align}
\phi_{n}(x, z) &= (1 - x - z^2(1 + x)) \phi_{n - 1}(x, z) - x z\begin{vmatrix}
xz & 0 & 0 & \dots \\
z & 1 - x & z & \dots \\
z^2 &  z & 1 - x & \dots \\
\vdots & \vdots & \vdots & \ddots
\end{vmatrix} \\
&= (1 - x - z^2(1 + x)) \phi_{n - 1}(x, z) - x^2 z^2 \phi_{n - 2}(x, z). \label{eq:kms_recurrence}\qedhere
\end{align}
\end{proof}
Since the transition probabilities are all the same for a given $N$, Eqs.~\eqref{eq:fund_R_eqs}, are invariant under interchange of any pair $(i_1, j_1, k_1)$, $(i_2, j_2, k_2)$. Therefore $R_{\text{sym}}(\lambda) \defn R_{i, j}^{(k)}(\lambda)$ satisfies a quadratic equation,
\begin{equation}
    R_{\text{sym}} = \frac{\lambda}{2(N-1)}\left[1 + (N - 2)R_{\text{sym}} + (N - 1)R_{\text{sym}}^2 \right].
\end{equation}
Substituting $\lambda = 1$ and taking the root such that $R_{\text{sym}}(1) \in (0, 1)$ gives
\begin{equation}
    R_{\text{sym}}(1) = \frac{1}{N - 1},\quad R_{\text{sym}}'(1) = \frac{2}{N - 2}, \quad R_{\text{sym}}''(1) = \frac{4(N^2 - 2)}{(N - 2)^3}.
\end{equation}
In the $|\cdot|$ metric, $\matinds{B}{i}{j}(k; \lambda, z) = (1 - \delta_{i, j})z R_{\text{sym}}(\lambda)$. Since $\ismat{B}(k; \lambda, z)$, does not depend on $k$, we will write $\ismat{H}(\lambda, z) \defn \ismat{B}(k; \lambda, z)$. By Lemma~\ref{lemma:det_kms},
\begin{equation}
    \text{det}\left[\idmat \pm \ismat{H}(\lambda, z)\right] = ( \pm z R_{\text{sym}}(\lambda))^N \phi_{N}\left(1 \mp \frac{1}{z R_{\text{sym}}(\lambda)} , 1 \right).
\end{equation}
Hence, we compute $h(\lambda, z) = \text{det}\left[\idmat + \ismat{H}(\lambda, z) \right] \text{det}\left[\idmat -  \ismat{H}(\lambda, z) \right]$. This can be computed up to terms of order $(z - 1)^2$ by substituting $z = 1$ into Eq.~\eqref{eq:kms_result}, solving the resulting linear recurrence relation and applying implicit differentiation. Via Eqs.~\eqref{eq:gamma_intro} and \eqref{eq:sigma_intro} we find the constants $\gamma_{\text{sym}}, \sigma^2_{\text{sym}}$ by expanding $\phi_N(x, z)$ in powers of $(z - 1)$.

In the $|\cdot|_{\metrictwo}$ metric, $\matinds{B}{i}{j} = (1 - \delta_{i, j}) z^{|i - j|} R_{\text{sym}}$. By Lemma  \ref{lemma:det_kms},
\begin{equation} \label{eq:detIminusB_W_metric}
 \det \left[\idmat \pm \ismat{H}(\lambda, z) \right] = (\pm R_{\text{sym}}(\lambda))^N \phi_{N}\left(1 \mp \frac{1}{R_{\text{sym}}(\lambda)}, z\right).
\end{equation}
We again have $h(\lambda, z) = \text{det}\left[\idmat + \ismat{H}(\lambda, z) \right]\, \text{det}\left[\idmat -  \ismat{H}(\lambda, z) \right]$ and the same arguments as above give $\gamma_{\text{sym}, \metrictwo}, \sigma^2_{\text{sym}, \metrictwo}$.

\begin{remark} \label{remark:symmetric_examples}
In the $|\cdot|$ metric, due to the symmetry of the problem, the process $ \{ |W_n|  \}_{n \geq 0}$ is a lazy nearest neighbor walk on the non-negative integers with transition probabilities
\begin{equation}
    P(|W_{n + 1}| = y \st |W_{n}| = x) = \begin{cases}
    1 & \text{if } x = 0 \text{ and } y = 1\\
    \frac{N - 1}{2(N - 1)} & \text{if } x \neq 0 \text{ and } y = x + 1 \\
    \frac{N - 2}{2(N - 1)} & \text{if } x \neq 0 \text{ and } y = x \\
    \frac{1}{2(N - 1)} & \text{if } x \neq 0 \text{ and } y = x - 1 \\
    0 & \text{otherwise}.
    \end{cases}
\end{equation}
This process is transient so Eqs.~\eqref{eq:symmetric_natural_metric_gamma} and \eqref{eq:symmetric_natural_metric_sigma} follow by direct computation.
\end{remark}

\section{Outline of proof}
\label{sec:outline_of_proof}

Our proof strategy uses the double generating function method of Sawyer and Steger (\cite{Sawyer1987}, Theorem~2.2):
\begin{theorem}[Sawyer and Steger] \label{thm:sawyer}
Let $\{Y_n\}_{n \geq 0}$ be a sequence of non-negative random variables and suppose that we can write for some $\delta > 0$
\begin{equation} \label{eq:sawyer_double}
    \genf{G}(\lambda, z) \defn E\left[\sum_{n = 0}^{\infty} z^{Y_n} \lambda^n \right] = \frac{C(\lambda, z)}{g(\lambda, z)} \quad \text{for } \lambda, z \in (1 - \delta, 1),
\end{equation}
where $C(\lambda, z), g(\lambda, z)$ can be extended as analytic functions to the regions $1 - \delta < |\lambda| < 1 + \delta$, $|z - 1| < \delta$ in the complex plane, and $C(1, 1) \neq 0$. Let
\begin{subequations}
\begin{align}
    \mu &\defn \frac{\partial_z g(1, 1)}{\partial_\lambda g(1, 1)},  \\
    \nu^2 &\defn \frac{\partial_z^2g(1, 1) + \partial_z g(1, 1) - 2 \mu \, \partial_{z, \lambda}^2 g(1, 1) + \mu^2 \, \left(\partial_\lambda^2g(1, 1) + \partial_\lambda g(1, 1) \right) }{\partial_\lambda g(1, 1)}.
\end{align}
\end{subequations}
Then
\begin{subequations}
\begin{align}
    \lim_{n \to \infty} \frac{Y_{n}}{n} &= \mu \quad a.s. \\
    \frac{Y_{n} - \mu n}{\sqrt{n}} &\to \mathcal{N}(0, \nu^2) \quad \text{in law}.
\end{align}
\end{subequations}
\end{theorem}
We will apply Theorem \ref{thm:sawyer} with $Y_n = |W_n|_{\metric}$. In order to understand the expectation in Eq.~\eqref{eq:sawyer_double} we introduce a family of stopping times. For $x \in \arrowset$ let
\begin{equation} \label{eq:T_stopping_time_def}
    T(m, x) \defn \inf_{k \geq 0} \{W_{m + k} = W_{m} x \};
\end{equation}
note that these can be $\infty$. We define a set of $2N(N - 1)$ generating functions, one for each element of $\genset{N}$ as
\begin{equation} \label{eq:R_fund_all}
    R_{i, j}^{(k)}(\lambda) \defn E_{e_i}\left[\lambda^{T\left(0, A_{i , j}^{(k)} \right)} \right],
\end{equation}
defined for those $\lambda \in \isset{C}$ where the expectation is finite. As we will show, these generating functions are the functions described in Proposition~\ref{prop:R_properties}. To do this, we will show that the $R_{i, j}^{(k)}(\lambda)$ introduced in Eq.~\eqref{eq:R_fund_all} uniquely satisfy Properties~\ref{eq:R_property_radius}, \ref{eq:R_property_physical} and \ref{eq:R_property_equations} of Proposition~\ref{prop:R_properties}.
\begin{remark} \label{remark:R_properties_immediate}
Since each $R_{i, j}^{(k)}(\lambda)$ is a power series whose coefficients are non-negative numbers summing to at most 1, we immediately see that they satisfy weaker versions of Properties~\ref{eq:R_property_radius} and \ref{eq:R_property_physical}. In particular, if $\lambda \in [0, 1]$ then $R_{i, j}^{(k)}(\lambda) \in [0, 1]$ and each $R_{i, j}^{(k)}(\lambda)$ is complex analytic in $\lambda \in D_1$.
\end{remark}
By a first step analysis based on the Markov property, we show the following proposition in Section~\ref{sec:gen_funcs_eqn}:
\begin{proposition} \label{prop:R_property_equations}
The set of functions $R_{i, j}^{(k)}(\lambda)$ introduced in Eq.~\eqref{eq:R_fund_all} satisfy Eq.~\eqref{eq:fund_R_eqs} for $\lambda \in [0, 1]$.
\end{proposition}
In Section \ref{sec:transience}, we use Perron--Frobenius theory to prove:
\begin{proposition} \label{prop:transience}
The Markov chain $\{W_n\}_{n \geq 0}$ is transient.
\end{proposition}
As a direct corollary, we have:
\begin{corollary} \label{cor:r1_less_than_1}
The set of functions $R_{i, j}^{(k)}(\lambda)$ introduced in Eq.~\eqref{eq:R_fund_all} satisfy $R_{i, j}^{(k)}(1) < 1$.
\end{corollary}
\begin{proof}
Proposition~\ref{prop:transience} implies that $R_{i, j}^{(k)}(1) = P_{e_i}(T(0, A_{i, j}^{(k)})<\infty) < 1$.
\end{proof}
In Section~\ref{sec:unique_root} we demonstrate using techniques from the analysis of branching processes that
\begin{proposition} \label{prop:Rs_unique}
For a given $\lambda \in [0, 1]$, the only solution to Eq.~\eqref{eq:fund_R_eqs} satisfying Property~\ref{eq:R_property_physical} is given by the $R_{i, j}^{(k)}(\lambda)$ introduced in Eq.~\eqref{eq:R_fund_all}.
\end{proposition}
For $x, y \in \arrowset$, define the generating function
\begin{equation} \label{eq:visit_generating_function}
    \genf{S}(x, y; \lambda) \defn \sum_{n = 0}^{\infty} P_x(W_{n} = y)\, \lambda^n .
\end{equation}
Next, define the generating function associated with first visits from an arbitrary word to be
\begin{equation} \label{eq:stopping_generating_function}
    \genf{R}(x, y; \lambda) \defn \sum_{n = 0}^{\infty} P_x(T(0, y) = n)\, \lambda^n .
\end{equation}
Note that these functions are identically zero unless $x, y \in \arrowset_i$ for some $1 \leq i \leq N$ by Lemma~\ref{lemma:intermediate_words}. For any $w \in \arrowset_i$, we have $\genf{S}(e_i, w; \lambda) = \genf{R}(e_i, w; \lambda)\, \genf{S}(w, w; \lambda)$ using the strong Markov property with the first hitting time of $w$.  In Section \ref{sec:proof_of_roc}, by obtaining an exponential bound on $P_{e_i}(W_n = e_i)$, we show the following proposition:
\begin{proposition} \label{prop:radii_of_convergence}
The functions $\genf{S}(x, y;\lambda)$ and $\genf{R}(x, y; \lambda)$ introduced in Eqs.~\eqref{eq:visit_generating_function} and \eqref{eq:stopping_generating_function}, respectively, have radii of convergence strictly greater than 1.
\end{proposition}
These ingredients allow us to prove Proposition~\ref{prop:R_properties}.
\begin{proof}[Proof of Proposition~\ref{prop:R_properties}]

We will show that the $R_{i, j}^{(k)}$ defined in Eq.~\eqref{eq:R_fund_all} are the unique functions satisfying Properties~\ref{eq:R_property_radius}--\ref{eq:R_property_equations}.

Since $R_{i, j}^{(k)}(\lambda) = \genf{R}(e_i, A_{i, j}^{(k)}; \lambda)$ for each $(i, j, k)$, by Proposition~\ref{prop:radii_of_convergence} there is an $\epsilon > 0$ such that $R_{i, j}^{(k)}$ is complex analytic in $D_{1 + \epsilon}$ as required for Property~\ref{eq:R_property_radius}. Property~\ref{eq:R_property_physical} is satisfied for $\lambda \in [0, 1)$ by definition and for $\lambda = 1$ by Corollary~\ref{cor:r1_less_than_1}. By Proposition~\ref{prop:R_property_equations} and Property~\ref{eq:R_property_radius}, the $R_{i, j}^{(k)}$ satisfy Eq.~\eqref{eq:fund_R_eqs} in $D_{1 + \epsilon}$ as required for Property~\ref{eq:R_property_equations}. Finally, by Proposition~\ref{prop:Rs_unique} and Property~\ref{eq:R_property_radius}, the $R_{i, j}^{(k)}$ are unique which completes the proof.
\end{proof}

We define
\begin{equation} \label{eq:double_gen_G}
    \genf{G}_{i}(\lambda, z) \defn E_{e_i}\left[\sum_{n = 0}^{\infty} z^{ |W_n|_{\metric} } \lambda^n \right].
\end{equation}
Let $\ismat{K}(\lambda, z)$ be the $2N \times 2N$ matrix whose blocks are
\begin{equation} \label{eq:A_mat_def}
\ismat{K}(\lambda, z) \defn \begin{pmatrix}
\zeromat & \ismat{B}(1; \lambda, z) \\
\ismat{B}(-1; \lambda, z) & \zeromat
\end{pmatrix},
\end{equation}
where $\ismat{B}(k; \lambda, z)$ is defined in Eq.~\eqref{eq:B_matrix_intro} and $\zeromat$ is the zero matrix.

For the proof of Theorem~\ref{thm:main}, we require the following two propositions whose proofs are postponed until Section~\ref{sec:proof_of_G}.
\begin{proposition} \label{prop:A_matrix_properties}
In the region $(\lambda, z) \in (0, 1] \times (0, 1]$, the matrix $\ismat{K}(\lambda, z)$ is irreducible and it has non-negative entries. In the region $(\lambda, z) \in (0, 1) \times (0, 1]$, the spectral radius of $\ismat{K}(\lambda, z)$ is strictly less than 1.
\end{proposition}
\begin{proposition} \label{prop:G_in_terms_of_A} Let $\isvec{s}(\lambda)$ and $\isvec{v}(i)$ be $N$-vectors whose entries satisfy $\vecinds{s}{j}(\lambda) = \genf{S}(e_j, e_j; \lambda)$,  $\vecinds{v}{j}(i) = \delta_{ij}$ and let $\bar{\isvec{s}}(\lambda)$ and $\bar{\isvec{v}}(\lambda)$ be the $2N$-vectors $\bar{\isvec{s}}(\lambda) = (\isvec{s}(\lambda), \isvec{s}(\lambda))$ and $\bar{\isvec{v}}(i) = (\isvec{v}(i), \isvec{v}(i))$. Then, we have
\begin{equation} \label{eq:G_in_terms_of_A}
\genf{G}_i(\lambda, z)  = \genf{S}(e_i, e_i; \lambda) + \sum_{d = 1}^{\infty}  \bar{\isvec{v}}\transpose(i) \, \ismat{K}(\lambda, z)^d\,  \bar{\isvec{s}}(\lambda),
\end{equation}
where $\genf{G}_i(\lambda, z)$ is defined in Eq.~\eqref{eq:double_gen_G} and $\ismat{K}(\lambda, z)$ is defined in Eq.~\eqref{eq:A_mat_def}.
\end{proposition}
We are now ready to prove Theorem~\ref{thm:main}.
\begin{proof}[Proof of Theorem~\ref{thm:main}]
It is enough to show Theorem~\ref{thm:main} with initial condition $W_0 = e_i$ for each $1 \leq i \leq N$. We will apply Theorem~\ref{thm:sawyer} with $Y_n = |W_n|_{\metric}$. Then, $\genf{G}_i$ in Eq.~\eqref{eq:double_gen_G} is $\genf{G}$ in Eq.~\eqref{eq:sawyer_double}.

By Proposition~\ref{prop:A_matrix_properties}, there is an $\epsilon > 0$ such that the geometric series in Eq.~\eqref{eq:G_in_terms_of_A} converges for $\lambda, z \in (1 - \epsilon, 1)$. Therefore, in this region, we have
\begin{equation}
    \sum_{d = 0}^{\infty}  \bar{\isvec{v}}\transpose(i) \, \ismat{K}(\lambda, z)^d\,  \bar{\isvec{s}}(\lambda) = \bar{\isvec{v}}\transpose(i) \, (\idmat - \ismat{K}(\lambda, z))^{-1} \,  \bar{\isvec{s}}(\lambda) = \frac{\isvec{v}\transpose(i) \, \text{adj}[ \idmat - \ismat{K}(\lambda, z) ] \, \bar{\isvec{s}}(\lambda)}{\text{det}[ \idmat - \ismat{K}(\lambda, z) ]},
\end{equation}
where $\text{adj}[\cdot]$ is the adjugate matrix. Applying this to Eq.~\eqref{eq:G_in_terms_of_A} gives
\begin{equation}
    \genf{G}_{i}(\lambda, z) = \frac{C_i(\lambda, z)}{\det[\idmat - \ismat{K}(\lambda, z)]},
\end{equation}
where $C_i(\lambda, z)$ depends on $\text{adj}[\idmat - \ismat{K}]$ and on $\genf{S}(e_i, e_i;\lambda)$. Therefore, we have that $C_i(\lambda, z)$ and $\det[\idmat - \ismat{K}(\lambda, z)]$ are both polynomial functions of the $R_{i, j}^{(k)}(\lambda)$, $\genf{S}(e_i, e_i;\lambda)$. By Proposition~\ref{prop:radii_of_convergence}, these functions are analytic for $\lambda \in D_{1 + \epsilon}$ with a possibly smaller $\epsilon > 0$. In addition, $C_i(\lambda, z)$ and $\det[\idmat - \ismat{K}(\lambda, z)]$ both depend on $z$ through a finite number of positive powers and hence they are complex analytic in a neighborhood of $z = 1$. It follows that $C_i(\lambda, z)$ and $\det[\idmat - \ismat{K}(\lambda, z)]$ are complex analytic in the region $\lambda \in D_{1 + \epsilon}, |z - 1| < \epsilon$.

It remains to show that $C_i(1, 1) \neq 0$. By Proposition~\ref{prop:A_matrix_properties}, $\ismat{K}(\lambda, z)$ is non-negative and irreducible so, by Perron--Frobenius theory, there is a simple eigenvalue $\mu_{\text{PF}}(\lambda, z)$ of $\ismat{K}(\lambda, z)$ equal to the spectral radius of $\ismat{K}(\lambda, z)$. Also by Proposition~\ref{prop:A_matrix_properties}, we have $\mu_{\text{PF}}(\lambda, z) < 1$ when $(\lambda, z) \in [0, 1) \times (0, 1]$. We can write
\begin{equation} \label{eq:g_pole}
    \genf{G}_i(\lambda, 1) = \frac{1}{1 - \lambda} = \frac{C_i(\lambda, 1)}{\text{det}[\idmat - \ismat{K}(\lambda, 1) ]}.
\end{equation}
Therefore, we must have $\mu_{\text{PF}}(1,1) \geq 1$. By Proposition~\ref{prop:radii_of_convergence}, the entries of $\ismat{K}(\lambda, 1)$ are complex analytic in a neighborhood of $\lambda = 1$. Therefore, $\mu_{\text{PF}}(\lambda, 1)$ is also a complex analytic in a neighborhood of $\lambda = 1$ \cite{kato2013perturbation} and so we have $\mu_{\text{PF}}(1,1) = 1$. The characteristic polynomial of $\ismat{K}(\lambda, 1)$ satisfies
\begin{equation}
    \text{det}[x\, \idmat - \ismat{K}(\lambda, 1)] = (x - \mu_{\text{PF}}(\lambda, 1))\,k(x, \lambda),
\end{equation}
where $k(x, \lambda)$ is a polynomial in $x$ with $k( \mu_{\text{PF}}(\lambda, 1), \lambda) \neq 0$.  This implies that the function $\text{det}[\idmat - \ismat{K}(\lambda, 1) ]$ has a simple zero at $\lambda = 1$. Hence, by Eq.~\eqref{eq:g_pole}, we have that $C_i(1, 1) \neq 0$.  We conclude that Theorem \ref{thm:sawyer} applies to the random variables $\{|W_{n}|_{\metric}\}_{n \geq 0}$ with $g(\lambda, z) = \text{det}[\idmat - \ismat{K}(\lambda, z)]$. Referring to Eq.~\eqref{eq:A_mat_def}, this can also be written in the form
\begin{equation}
    g(\lambda, z) = \text{det}[\idmat - \ismat{B}(1; \lambda, z)\,\ismat{B}(-1; \lambda, z)].
\end{equation}
This completes the proof of Theorem~\ref{thm:main}.
\end{proof}

\section{Proofs of the main steps}
\label{sec:proofs}

The following subsections contain the proofs of the propositions for the proof of Theorem~\ref{thm:main}.

\subsection{Proof of Proposition~\ref{prop:R_property_equations}}
\label{sec:gen_funcs_eqn}

\begin{proof}[Proof of Proposition~\ref{prop:R_property_equations}]
By Lemma~\ref{lemma:intermediate_words}, we have
\begin{equation} \label{eq:translational_symmetry_funds}
P(T(m, xy) = \ell \st W_m = x) = P_{e_{\target(x)}}(T(0, y) = \ell ).
\end{equation}
We can assume that $(i, j, k) = (1, 2, 1)$. Conditioning on possible values of $W_1$ gives
\begin{subequations}
\begin{align}
    E_{e_1}[\lambda^{T( 0, A_{1 , 2}^{(1)} )}] &= p_{1, 2}^{(1)} E_{e_1}[\lambda^{T ( 0, A_{1 , 2}^{(1)} )} \st W_{1} = A_{1 , 2}^{(1)}] \label{eq:fund_proof_1} \\
    &+  \sum_{m \neq 1, 2 } p_{1, m}^{(1)}  E_{e_1}[\lambda^{T( 0, A_{1 , 2}^{(1)} )} \st W_{1} = A_{1 , m}^{(1)} ] \label{eq:fund_proof_2} \\
    &+ \sum_{m \neq 1} p_{1, m}^{(-1)}  E_{e_1}[\lambda^{T( 0, A_{1 , 2}^{(1)} )} \st W_{1} = A_{1 , m}^{(-1)} ]. \label{eq:fund_proof_3}
\end{align}
\end{subequations}
Let us consider these terms line-by-line. In line~\eqref{eq:fund_proof_1}, on the event $W_1 = A_{1 , 2}^{(1)}$ we have $T(0, A_{1 ,2}^{(1)} ) = 1$, thus $E[\lambda^{T ( 0, A_{1 , 2}^{(1)} )} \st W_{1} = A_{1 , 2}^{(1)}] = \lambda$.

Consider any term in the sum on line~\eqref{eq:fund_proof_2}. On the event $\{W_1 = A_{1 , m}^{(1)} \}$ we have $T(0, A_{1 , 2}^{(1)}) = 1 + T(1, A_{m , 2}^{(1)})$ by definition. Hence, by Eq.~\eqref{eq:translational_symmetry_funds},
\begin{equation}
     E_{e_1}[\lambda^{T( 0, A_{1 , 2}^{(1)} )} \st W_{1} = A_{1 , m}^{(1)} ] =  E_{e_m}[\lambda^{1 + T(0, A_{m , 2}^{(1)})}] = \lambda R_{m, 2}^{(1)}(\lambda).
\end{equation}
Consider any term in the sum on line \eqref{eq:fund_proof_3}. If the process is at $A_{1, m}^{(-1)}$ with $m \neq 1$ and visits $A_{1,  2}^{(1)}$ at a later time, then by Lemma~\ref{lemma:intermediate_words} it must visit $e_1$ before visiting $A_{1,  2}^{(1)}$. On the event $\{W_1 = A_{1,  m}^{(-1)} \}$ the first visit to $e_1$ happens at step $1 + T(1, A_{m,  1}^{(-1)})$. Conditioning on the value of $T(0, A_{m, 1}^{(-1)})$ and applying Eq.~\eqref{eq:translational_symmetry_funds} gives
\begin{equation}
E_{e_1}[\lambda^{T(0, A_{1 , 2}^{(1)})} \st W_1 = A_{1 , m}^{(-1)} ] = E_{e_m}[\lambda^{1 + T(0, A_{m , 1}^{(-1)} A_{1 , 2}^{(1)} )}] = \lambda R_{m, 1}^{(-1)}(\lambda) R_{1, 2}^{(1)}(\lambda).
\end{equation}
Collecting all these terms gives Eq.~\eqref{eq:fund_R_eqs}.
\end{proof}

\subsection{Proof of Proposition \ref{prop:transience}}
\label{sec:transience}

\begin{proof}[Proof of Proposition~ \ref{prop:transience}]
Suppose that $\{W_n\}_{n \geq 0}$ is recurrent for a given initial condition. Then, by Lemma~\ref{lemma:intermediate_words}, it is recurrent for any initial condition. Moreover, for a given initial condition $e_i$, the hitting times $T(0, A_{i,j}^{(k)})$ are almost surely finite, hence $R_{i, j}^{(k)}(1) = 1$ for each $(i, j ,k)$.

Let $D_{i, j}^{(k)}(\lambda) \defn \text{d}R_{i, j}^{(k)}/\text{d} \lambda $; we have $0 < D_{i, j}^{(k)}(\lambda) < \infty$ for $0 < \lambda < 1$ (see Remark~\ref{remark:R_properties_immediate}). Define the $2N(N - 1)$-vectors $\isvec{d} \defn (D_{i, j}^{(k)})$ and $\isvec{r} \defn (R_{i, j}^{(k)})$. Differentiating Eq.~\eqref{eq:fund_R_eqs} with respect to $\lambda$ gives the linear system
\begin{equation} \label{eq:M_linear_system}
    \isvec{d} = \lambda^{-1}\isvec{r} + \ismat{M}(\lambda) \isvec{d},
\end{equation}
where $\ismat{M}(\lambda)$ is a $2N(N - 1) \times 2N(N - 1)$ matrix whose entries are
\begin{equation} \label{eq:M_mat_def}
    M(\lambda)_{(i, j, k)}^{(i', j', k')} = \begin{cases}
    \lambda  p_{i, i'}^{(k)} & i' \neq j,\ j' = j,\ k' = k \\
    \lambda p_{i, i'}^{(-k)} R_{i, j}^{(k)}(\lambda) & j' = i,\ k' \neq k \\
    \lambda \sum_{\ell \neq i } p_{i, \ell}^{(-k)} R_{\ell, i}^{(-k)}(\lambda) & i = i',\ j = j',\ k = k'.
    \end{cases}
\end{equation}

By Remark~\ref{remark:R_properties_immediate}, $\ismat{M}(\lambda)$ extends continuously to $\lambda \in [0, 1]$.

By Lemma~\ref{lemma:M_primitive} below, the matrix $\ismat{M}(\lambda)$ is primitive for $\lambda \in (0, 1]$, and hence has Perron--Frobenius eigenvalue $\mu(\lambda) >  0$. There is a corresponding eigenvector $\isvec{v}(\lambda)$ with positive entries, and all other eigenvalues of $\ismat{M}(\lambda)$ are smaller than $\mu(\lambda)$ in norm (see for example \cite{seneta2006non}). Since $\isvec{r} > \zerovec$, by Eq.~\eqref{eq:M_linear_system} we have $\isvec{d} >\ismat{M} \isvec{d}$ and multiplying this vector inequality from the left with $\isvec{v}\transpose$ gives
\begin{equation}
\isvec{v}\transpose \isvec{d} > \isvec{v}\transpose \ismat{M} \isvec{d} = \mu \isvec{v}\transpose \isvec{d}.
\end{equation}
This shows that $\mu(\lambda)<1$ for $\lambda \in (0, 1)$ by the positivity of $\isvec{d}$. We will show that $\mu(\lambda_0) > 1$ for some $\lambda_0 \in (0, 1)$ which proves the lemma by contradiction.

We will now show that there is a non-zero vector $\isvec{u} =(u_{i, j}^{(k)})$ that satisfies $(\ismat{M}(1) - \idmat) \isvec{u} = \zerovec$ and
\begin{equation} \label{eq:M_eigenvector_ansatz}
 u_{i, j}^{(k)} = (1 - \delta_{j, 1})\tilde{u}_{j}^{(k)} - (1 - \delta_{i, 1}) \tilde{u}_{i}^{(k)},
\end{equation}
for some nonzero numbers $\{\tilde{u}_{i}^{(k)} : 1 \leq i \leq N, k \in \{-1, 1\} \}$. We introduce the $N\times 2N$ matrix
\begin{subequations}
\begin{align}
\ismat{J} &\defn (x_{i, (\ell,m)})_{1\le i, \ell \le N, m\in \{-1,1\}}, \\
x_{i,(\ell,m)}
&\defn
- p_{i, \ell}^{(m)} (1-\delta_{\ell, 1}) (1-\delta_{\ell,i})+\delta_{i,\ell} (1-\delta_{i,1})\sum_{j\neq i} p_{i,j}^{(m)}.
\end{align}
\end{subequations}
Suppose that $\isvec{u}$ is of the form in Eq.~\eqref{eq:M_eigenvector_ansatz}. Then for each $(i,j,k)$ we have
\begin{subequations}
\begin{align}
\left[(\ismat{M}(1) - \idmat) \isvec{u} \right]_{(i, j, k)}&=
\sum_{\ell \neq i, j   } p_{i, \ell}^{(k)}  u_{\ell, j}^{(k)} + \sum_{\ell \neq i} p_{i, \ell}^{(-k)} u_{\ell, i}^{(-k)}  - u_{i, j}^{(k)} \sum_{\ell \neq i} p_{i, \ell}^{(k)} \label{eq:m1_minus_I_u}\\
&=
\sum_{\ell \neq i} \sum_{m = \pm 1} p_{i, \ell}^{(m)}\left[ (1 - \delta_{i, 1}) \tilde{u}_{i}^{(m)}- (1 - \delta_{\ell, 1}) \tilde{u}_\ell^{(m)} \right]\\
&=\left[\ismat{J} \tilde{\isvec{u}}\right]_i\,,
\end{align}
\end{subequations}
where we have used $R_{i, j}^{(k)}(1) = 1$ in Eq.~\eqref{eq:m1_minus_I_u}. Hence $\ismat{J} \tilde{\isvec{u}} = \zerovec$ implies $(\ismat{M}(1) - \idmat) \isvec{u} = \zerovec$ if $\isvec{u}$ is given by Eq.~\eqref{eq:M_eigenvector_ansatz}. Since $\ismat{J}$ has dimensions $N\times 2N$, there is a non-trivial vector $\tilde{\isvec{u}}$ in its null-space. Then the vector $\isvec{u}$ defined via Eq.~\eqref{eq:M_eigenvector_ansatz} satisfies $(\ismat{M}(1) - \idmat) \isvec{u} = \zerovec$ which shows that 1 is an eigenvalue of  $\ismat{M}(1)$. Since  $u_{i, j}^{(k)} = - u_{j, i}^{(k)}$ for all $i,j,k$, the entries of the  corresponding eigenvector cannot be all positive, so 1 cannot be the Perron--Frobenius eigenvalue of $\ismat{M}(1)$. This implies $\mu(1)>1$.

Since $\mu(\lambda)$ is a simple root of the characteristic polynomial of $\ismat{M}(\lambda)$, it is continuous in the coefficients of that polynomial \cite{zedek1965}. These coefficients are continuous functions of $\lambda$ on $[0, 1]$. Hence, $\mu(\lambda)$ is a continuous function of $\lambda$, so there exists a $\lambda_0 \in (0, 1)$ such that $\mu(\lambda_0) > 1$, and we have the desired contradiction.
\end{proof}

\subsection{Proof of Proposition~\ref{prop:Rs_unique}}
\label{sec:unique_root}

The proof of Proposition~\ref{prop:Rs_unique} follows well-known techniques in branching processes, see Sevastyanov \cite{sevastyanov1971}.

\begin{proof}[Proof of Proposition~\ref{prop:Rs_unique}]

Fix a given $\lambda \in [0, 1]$. We begin considering Eq.~\eqref{eq:fund_R_eqs} which, in order to simplify the notation, we express as a $2N(N - 1)$-vector equation,
\begin{equation} \label{eq:fixed_point}
    \isvec{q} = \isvec{f}(\isvec{q}, \lambda).
\end{equation}
The vector entries are labeled by the multi-index $\ell = (i, j, k)$, with
\begin{equation} \label{eq:F_fixed_point_def}
    f_{\ell}(\isvec{q}, \lambda)=f_{i, j}^{(k)}(\isvec{q}, \lambda) = \lambda\left(p_{i, j}^{(k)} + \sum_{m \neq i, j } p_{i, m}^{(k)}  q_{m, j}^{(k)} + \sum_{m \neq i } p_{i, m}^{(-k)} q_{m, i}^{(-k)} q_{i, j}^{(k)}\right).
\end{equation}
Observe that each $\vecinds{f}{\ell}(\isvec{q}, \lambda)$ is a quadratic polynomial in $\isvec{q}$ whose coefficients are non-negative numbers summing to at most 1. Therefore, for $2N(N - 1)$-vectors $\isvec{x}, \isvec{y}$, we have
\begin{equation} \label{eq:f_non_decreasing}
    \zerovec \leq \isvec{x} < \isvec{y} \leq \onevec \implies \zerovec  < \isvec{f}(\isvec{x}, \lambda) < \isvec{f}(\isvec{y}, \lambda) \leq \onevec,
\end{equation}
where $\onevec$ is a vector with unit entries. Let $\{\isvec{a}_n(\lambda) \}_{n \geq 0}$ be a sequence defined recursively as
\begin{equation} \label{eq:a_seq_defn}
    \isvec{a}_{n + 1}(\lambda) = \isvec{f}(\isvec{a}_{n}(\lambda), \lambda), \quad \isvec{a}_0(\lambda) = \isvec{f}(\zerovec, \lambda).
\end{equation}
By Lemma~\ref{lemma:sequence_converges_to_r} and Corollary~\ref{cor:r1_less_than_1},
\begin{equation} \label{eq:a_iterates}
    (\isvec{a}_n(\lambda))_\ell \leq \sum_{m = 0}^{2^{n + 1} - 1}  P_{e_i}\left(T\left( 0, A_{i, j}^{(k)} \right) = m \right) \lambda^m \leq  R_{i, j}^{(k)}(\lambda) < 1, \quad n \geq 0.
\end{equation}
Let $\isvec{q}^\star(\lambda)$ be a vector whose entries are $\vecinds{q}{\ell}^\star(\lambda) = R_{i, j}^{(k)}(\lambda)$. By Eqs.~\eqref{eq:f_non_decreasing} and \eqref{eq:a_iterates} we have that $\{\isvec{a}_n(\lambda) \}_{n \geq 0}$ is a strictly increasing sequence bounded above by $\isvec{q}^\star(\lambda)$. Therefore, $\lim_{n \to \infty} \isvec{a}_n(\lambda) = \isvec{q}^\star(\lambda)$, where $\isvec{f}(\isvec{q}^\star, \lambda) = \isvec{q}^\star$ and $\zerovec < \isvec{q}^\star(\lambda) < \onevec$.

\

We will now show that $\isvec{q}^\star(\lambda)$ is the only solution to Eq.~\eqref{eq:fixed_point} satisfying $\zerovec < \isvec{q}(\lambda) < \onevec$. Suppose that there is an $\isvec{r}(\lambda) \neq \isvec{q}^\star(\lambda)$ such that $\isvec{r} = \isvec{f}(\isvec{r}, \lambda)$ and $\zerovec < \isvec{r}(\lambda) < \onevec$. Applying the function $\isvec{f}(\cdot, \lambda)$ repeatedly to both sides of the inequality $\zerovec <  \isvec{r}(\lambda)$ and using Eq.~\eqref{eq:f_non_decreasing} we get $\isvec{q}^\star(\lambda) \leq \isvec{r}(\lambda)$. We will drop the $\lambda$ dependence for the remainder of this proof. Draw the line $\isvec{z}(\theta) =\isvec{q}^\star + (\isvec{r} - \isvec{q}^\star)\theta$; there will be a point on this line $\tilde{\isvec{r}} \leq \onevec$ such that $\vecindstilde{r}{\ell} = 1$ for some $\ell$ and $\tilde{\theta} > 1$. We therefore have $\vecinds{f}{\ell}(\tilde{\isvec{r}}) \leq 1 = \vecindstilde{r}{\ell}$ by Eq.~\eqref{eq:f_non_decreasing}. Let $\varphi(\theta) = \vecinds{f}{\ell}(\isvec{z}(\theta)) - \vecinds{z}{\ell}(\theta)$, then we have
\begin{equation}
  \varphi(0) = \vecinds{f}{\ell}(\isvec{q}^\star) - \vecinds{q}{\ell}^\star = 0,
  \qquad
  \varphi(\tilde{\theta}) = \vecinds{f}{\ell}(\tilde{\isvec{r}}) - \vecindstilde{r}{\ell} \leq 0.
\end{equation}
By direct computation, and since $\isvec{f}_\ell$ is a quadratic polynomial with non-negative coefficients,
\begin{equation} \label{eq:phi_convexity}
\varphi''(\theta) = \sum_{m, n} (\vecinds{r}{m} - \vecinds{q}{m}^\star)(\vecinds{r}{n} - \vecinds{q}{n}^\star) \frac{\partial^2 \vecinds{f}{\ell}}{\partial\vecinds{z}{m} \partial \vecinds{z}{n}} \geq 0.
\end{equation}
Since $\vecinds{f}{\ell}$ is nonlinear, $\varphi''(\theta)$ is not identically zero so convexity gives $\varphi(\theta) < 0$ for $\theta \in (0, \tilde{\theta})$. In particular, $\varphi(1) < 0$, so $\vecinds{f}{\ell}(\isvec{r}) < \vecinds{r}{\ell}$: a contradiction.
\end{proof}

\subsection{Proof of Proposition~\ref{prop:radii_of_convergence}}
\label{sec:proof_of_roc}

In this section, we show that the functions $\genf{S}(x, y;\lambda)$ and $\genf{R}(x, y; \lambda)$ introduced in Eqs.~\eqref{eq:visit_generating_function} and \eqref{eq:stopping_generating_function}, respectively, have radii of convergence strictly greater than 1. This would follow from Proposition~8.1 in Lalley (2001) \cite{Lalley2001}, however the proof of this proposition is incomplete. A correction was provided to us by the author via personal communication and he kindly allowed us to reproduce the corrected proof here, adapted to our case.

To prove Proposition~\ref{prop:radii_of_convergence}, we require the following Lemma which is a consequence of the Azuma--Hoeffding inequality for bounded submartingales \cite{freedman1975}.
\begin{lemma} \label{lemma:large_deviations}
Let $\xi_1, \xi_2, \dots$ be a sequence of Bernoulli random variables adapted to a filtration $\{\mathcal{F}_n\}_{n \geq 0}$. Assume that there exists $p > 0$ such that, for every $n \geq 1$,
\begin{equation} \label{eq:bernoulli_condition}
P(\xi_{n + 1} = 1 \st \mathcal{F}_n) \geq p.
\end{equation}
Then for every $\alpha < p$, there exist $\beta < 1$ and $C < \infty$ such that $\forall n \geq 1$,
\begin{equation}
P\left(\sum_{i = 1}^{m} \xi_{i} \leq \alpha m \text{ for some } m\geq n\right) \leq C \beta^n .
\end{equation}
\end{lemma}

\begin{proof}[Proof of Proposition~\ref{prop:radii_of_convergence}]
Fix any $1 \leq i \leq N$. For any $x, y \in \arrowset_i$ we have $P_x(W_{n} = y) \geq P_x(T(0, y) = n)$ so it will be sufficient to show that the radius of convergence of $\genf{S}(x, y; \lambda)$ is strictly greater than 1; we begin by showing that that the radii of convergence of the $\genf{S}(x, y; \lambda)$ are all the same.

Let $x, x', y,y' \in \arrowset_i$. By Lemma~\ref{lemma:intermediate_words}, there exists a positive-probability path from $x$ to $y$ that passes through $x'$ then $y'$ on the way. Suppose that the shortest path from $x$ to $x'$ has $\ell_1$ steps, from $y$ to $y'$ has $\ell_2$ steps, and from $x'$ to $y'$ has $n'$ steps. Let $k = \ell_1 + \ell_2$, then by the Markov property we have that, for all $n \geq n'$,
\begin{equation}
    P_x(W_{n + k} = y) \geq P_x(W_{n + \ell_1 + \ell_2} = y, W_{n + \ell_1} = y', W_{\ell_1} = x' ) \geq \epsilon P_{x'}(W_n = y'),
\end{equation}
where $\epsilon > 0$ is independent of $n$. By the same argument, there exists a $k' \geq 0$ and $\epsilon' >  0$ such that, for sufficiently large $n$,
\begin{equation}
P_{x'}(W_{n + k'} = y') \geq \epsilon' P_{x}(W_n = y).
\end{equation}
Therefore, $\genf{S}(x, y; \lambda)$ and $\genf{S}(x', y'; \lambda)$ have the same radii of convergence.

We will now show that $\genf{S}(e_i, e_i; \lambda)$ has radius of convergence strictly greater than 1.  To do this, we will show that there are constants $C < \infty$ and $\beta < 1$ such that $P_{e_i}(W_n = e_i) \leq C \beta^n$. We will write $e \defn e_i$ for the remainder of this proof.

By Assumptions~\ref{ass:append_A_each_step}, \ref{ass:prob_form} and Lemma~\ref{lemma:intermediate_words}, the quantity
\begin{equation*}
    c(w) \defn P_{w}(|W_{n}| > |w| \text{ for all } n > 0)
\end{equation*}
takes one of a finite number of possible values (since, if $w$ is non-empty, $c(w)$ only depends on the last generator in the reduced representation of $w$.) By Proposition~\ref{prop:transience}, our process is transient and hence $c(A_{i, j}^{(k)}) > 0$ for at least one $(i, j, k)$. For any word $w$, we can append at most four generators at the end of $w$ to produce a word ending in $A_{i, j}^{(k)}$ in a way that the length of the word strictly increases during this process. Hence by the Markov property, there exists a $q > 0$ such that
\begin{equation} \label{eq:escape_prob}
P_w(|W_{n}| > |w|  \text{ for all } n > 0) \geq q \quad \forall w \in \arrowset.
\end{equation}
We fix $m \geq 1$ to be the smallest integer such that $m q > 1- q$.
We define the stopping times $\tau_k$ inductively such that $\tau_0 = 0$ and
\begin{equation}
\tau_{k + 1} \defn \min_{n > \tau_k}\{|W_{n}| - |W_{\tau_k}| \in\{-1, 0, m \} \}.
\end{equation}
By the transience of our process, $\tau_k < \infty$ almost surely and $W_{\tau_k}$ is well-defined.

Consider the event $\{W_n = e\}$ and take $j = \max\{k: \tau_k < n \}$.
By Assumptions~\ref{ass:append_A_each_step} and \ref{ass:prob_form}, we have $|W_{m + 1}| - |W_m| \in \{-1, 0, 1\}$ for each $m \geq 0$.  It follows that $|W_{\tau_j}| \leq 1$ since otherwise there would be a $\tau_j < n' < n$ such that $|W_{n'}| -|W_{\tau_j}|\in \{-1,0\}$, so $\tau_{j + 1} \leq n' < n$ which is a contradiction. Therefore, on the event $\{W_n = e\}$ we have $|W_{\tau_j}|\in\{0,1\}$ and $\tau_{j+1}=n$.

Let $\gamma \in (0, 1)$, then partitioning $\{W_n = e\}$ using the event $\{ \tau_{k} = n \text{ for some } k \geq \gamma n \}$ and its compliment we get the upper bound
\begin{equation} \label{eq:return_to_e_bound}
    P_e(W_n = e) \leq P(\tau_{\lceil \gamma n\rceil} > n) + P_e(W_{\tau_k} = e \text{ for some } k \geq \gamma n),
\end{equation}
where $\lceil \cdot \rceil$ is the ceiling function. We will show that there is a $\gamma$ such that each term in Eq.~\eqref{eq:return_to_e_bound} has an exponential bound.

Define the Bernoulli random variables $\{\xi_k\}_{k \geq 1}$ by
\begin{equation}
\xi_{k} \defn \begin{cases}
1 & m k < \tau_j \leq m(k + 1) \text{ for some } j \geq 0\\
0 & \text{otherwise}.
\end{cases}
\end{equation}
These are adapted to the filtration $\mathcal{F}_{m(k + 1)}$, $k \geq 1$ where $\mathcal{F}_n$ is the $\sigma$-algebra generated by the first $n$ steps of the process. By Assumption~\ref{ass:probs_positive}, we have
\begin{equation} \label{eq:increase_by_1_lower_bound}
    P_w(|W_1|=|w|+1) > 0,  \quad \forall w \in \arrowset.
\end{equation}
By Eq.~\eqref{eq:increase_by_1_lower_bound} and the Markov property, there exists an $\alpha > 0$ such that
\begin{equation} \label{eq:lemma_condition_xi}
P_e(\xi_{k + 1} = 1  \st \mathcal{F}_{mk}) \geq \alpha, \quad \forall k \geq 1.
\end{equation}
Suppose that $\tau_\ell > mb $ for some $\ell, b \geq 1$. If $\sum_{k = 1}^{b} \xi_{k} > \ell$, then there would be at least $\ell$ distinct intervals of length $m$ up to time $mb$ containing a $\tau_k$, which would imply that $\tau_\ell \leq mb$. Therefore, we must have
\begin{equation} \label{eq:tau_sum_condition}
P(\tau_\ell > mb) \leq P\left(\sum_{k = 1}^{b} \xi_{k} \le \ell \right).
\end{equation}
Consider the first term in Eq.~\eqref{eq:return_to_e_bound} and take $\gamma = {\alpha}/{(2m)}$. By Eqs.~\eqref{eq:lemma_condition_xi}, \eqref{eq:tau_sum_condition} and Lemma~\ref{lemma:large_deviations}, we have that there exist constants $C < \infty$ and $0 < \beta < 1$ such that
\begin{equation} \label{eq:k_lesser_exponential}
    P(\tau_{\lceil \gamma n \rceil} > n) \leq P\left(\sum_{k = 1}^{\ell} \xi_{k} \leq (\alpha/2)\ell \text{ for some } \ell \geq n/m \right) \leq C \beta^n, \quad \forall n \geq 0.
\end{equation}
Next, define the Bernoulli random variables $\{\zeta_k\}_{k \geq 1}$ by
\begin{equation}
\zeta_{k + 1} \defn \begin{cases}
1 & |W_{\tau_{k + 1}}|  - |W_{\tau_k}|  = m \\
0 & |W_{\tau_{k + 1}}|  - |W_{\tau_k}|  \in \{-1, 0\}
\end{cases}.
\end{equation}
These are adapted to the filtration $(\mathcal{F}_{\tau_{k + 1}})_{k \geq 0}$. By Eq.~\eqref{eq:escape_prob} and the Markov property, we have
\begin{equation} \label{eq:zeta_lower_bound}
    P(\zeta_{k + 1} = 1 \st \mathcal{F}_{\tau_k}) \geq q, \quad  \forall k \geq 0.
\end{equation}
We fix $r < q$ such that
\begin{equation} \label{eq:Delta_def}
m r - (1 - r) \defn \Delta > 0.
\end{equation}
By Eq.~\eqref{eq:zeta_lower_bound} and Lemma \ref{lemma:large_deviations}, there exist constants $K < \infty$ and $0 < \delta < 1$ such that
\begin{equation}
P_e\left( \sum_{\ell = 1}^{k} \zeta_\ell \leq r k \text{ for some } k \geq n \right) \leq K \delta^n, \quad \forall n \geq 0.
\end{equation}
If $\sum_{\ell = 1}^{k} \zeta_\ell \geq r k$, then $|W_{\tau_k}| \geq  m rk - (1 - r)k = k \Delta$ by Eq.~\eqref{eq:Delta_def}. Therefore,
\begin{equation} \label{eq:length_of_W_exponential}
P_e\left(|W_{\tau_k}| < k \Delta \text{ for some } k \geq n\right) \leq  K \delta^n \quad \forall n \geq 0.
\end{equation}
Considering the second term in Eq.~\eqref{eq:return_to_e_bound}, by Eq.~\eqref{eq:length_of_W_exponential} we have that for all $n \geq 0$,
\begin{equation} \label{eq:k_greater_exponential}
    P_e(W_{\tau_{k}} = e \,\text{ for some } k \geq \gamma n) \leq P_e\left(|W_{\tau_k}| < k \Delta \text{ for some } k \geq \gamma n \right) \leq K (\delta^\gamma)^n ,
\end{equation}
where $0 < \delta^\gamma < 1$. Combining Eqs.~\eqref{eq:k_lesser_exponential} and \eqref{eq:k_greater_exponential} with Eq.~\eqref{eq:return_to_e_bound} gives that $P_e(W_n = e)$ has an exponential bound, which in turn implies that the radius of convergence of $\genf{S}(e, e;\lambda)$ is strictly greater than 1. Therefore, we also have that the radius of convergence of $\genf{S}(x,y;\lambda)$ is strictly greater than 1 for all $x, y \in \arrowset_i$. Since $i$ was arbitrary, we conclude that the radius of convergence of $\genf{S}(x,y;\lambda)$ is strictly greater than 1 for all $x, y \in \arrowset$, as required.
\end{proof}

\subsection{Proof of Proposition~\ref{prop:A_matrix_properties} and Proposition~\ref{prop:G_in_terms_of_A}}
\label{sec:proof_of_G}
We begin by establishing two lemmas we will need for the proofs of Propositions~\ref{prop:A_matrix_properties} and \ref{prop:G_in_terms_of_A}.
\begin{lemma} \label{lemma:first_in_terms_of_funds}
Suppose that $w_1, w_2\in \arrowset_i$ for some $1 \leq i \leq N$ and $w_1 \neq w_2$. Consider the reduced composition
\begin{equation}
    w_1^{-1} w_2= \prod_{\ell = 1}^{d} A_{i_\ell , j_\ell}^{(k_\ell)} \quad d \geq 1
\end{equation}
given by Lemma~\ref{lemma:plane_switching}. Then, we have
\begin{equation} \label{eq:first_in_terms_of_funds}
    \genf{R}(w_1, w_2; \lambda) = \prod_{i = 1}^{d} R_{i_\ell, j_\ell}^{(k_\ell)}(\lambda).
\end{equation}
\end{lemma}
\begin{proof}
This follows from Lemma~\ref{lemma:intermediate_words} and the strong Markov property applied at the stopping times when the process reaches the intermediate words $\prod_{i = 1}^{m} A_{i_m, j_m}^{(k_m)}$ for each $1 \leq m \leq d$.
\end{proof}
We introduce the set of words
\begin{equation} \label{eq:F_set_def}
    \isset{F}_{i,j}^{(k)}(d) \defn \{w \in \arrowset_i: |w|  = d,\ \target(w) = j,\ \mathfrak{K}(w) = k  \},
\end{equation}
where $\mathfrak{K}(w) = k$ if $|w| \geq 1$ and the first letter of $w$ in its reduced representation is of the form $A_{i, j}^{(k)}$. A representative member of $\isset{F}_{i,j}^{(k)}(d)$ is shown in Figure~\ref{fig:F_example}.
\begin{figure}[ht]
\centering
\begin{tikzpicture}

\coordinate (LLtr) at (-47mm,3mm);
\coordinate (LLtl) at (-53mm,3mm);

\coordinate (LLbr) at (-47mm,-3mm);
\coordinate (LLbl) at (-53mm,-3mm);

\coordinate (Ltr) at (-22mm,3mm);
\coordinate (Ltl) at (-28mm,3mm);

\coordinate (Lbr) at (-22mm,-3mm);
\coordinate (Lbl) at (-28mm,-3mm);

\coordinate (Mtl) at(-3mm,3mm);
\coordinate (Mtr) at(3mm,3mm);

\coordinate (Mbl) at(-3mm,-3mm);
\coordinate (Mbr) at(3mm,-3mm);

\coordinate (Rtr) at (28mm,3mm);
\coordinate (Rtl) at (22mm,3mm);

\coordinate (Rbr) at (28mm,-3mm);
\coordinate (Rbl) at (22mm,-3mm);

\coordinate (RRtl) at (47mm,3mm);
\coordinate (RRtr) at (53mm,3mm);

\coordinate (RRbl) at (47mm,-3mm);
\coordinate (RRbr) at (53mm,-3mm);

\node[] at (-50mm,0mm) {$1$};
\draw[black, thick] (-50mm,0mm) circle (4.1mm);

\node[] at (-25mm,0mm) {$2$};
\draw[black, thick] (-25mm,0mm) circle (4.1mm);

\node[] at (0mm,0mm) {$3$};
\draw[black, thick] (0mm,0mm) circle (4.1mm);

\node[] at (25mm,0mm) {$4$};
\draw[black, thick] (25mm,0mm) circle (4.1mm);

\node[] at (50mm,0mm) {$5$};
\draw[black, thick] (50mm,0mm) circle (4.1mm);

\draw[-{Latex[length=2mm, width=2mm]}, black, very thick] (LLtr) to[out=45,in=135] (Ltl);

\draw[-{Latex[length=2mm, width=2mm]}, black, very thick] (Lbr) to[out=-45,in=-135] (RRbl);

\draw[-{Latex[length=2mm, width=2mm]}, black, very thick] (RRtl) to[out=135,in=45] (Mtr);

\draw[-{Latex[length=2mm, width=2mm]}, black, very thick] (Mbr) to[out=-45,in=-135] (Rbl);

\end{tikzpicture}
\caption{One member of the set $\isset{F}_{1, 4}^{(1)}(4)$ in the $N = 5$ case.}
\label{fig:F_example}
\end{figure}
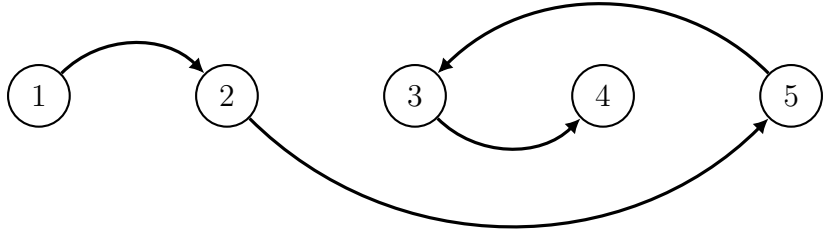
Let $\overline{\ismat{B}}(d, k; \lambda, z)$ be the matrix defined by the product
\begin{equation} \label{eq:B_prod_mat_def}
    \overline{\ismat{B}}(d, k; \lambda, z) \defn \prod_{i = 1}^{d} \ismat{B} \left( (-1)^{i + 1} k; \lambda, z \right),
\end{equation}
where $\ismat{B}(k; \lambda, z)$ is defined in Eq.~\eqref{eq:B_matrix_intro}. Then, we have the following lemma.
\begin{lemma} \label{lemma:B_in_terms_of_R}
The entries of $\overline{\ismat{B}}(d, k; \lambda, z)$ satisfy
\begin{equation} \label{eq:Bbar_in_terms_of_R}
\matindsbar{B}{i}{j}(d, k; \lambda, z) = \sum_{w \in   \isset{F}_{i,j}^{(k)}(d)} z^{|w|_{\metric}} \genf{R}(e_i, w; \lambda), \quad d \geq 1.
\end{equation}
\end{lemma}
\begin{proof}
The proof is by induction on $d$. The set $\isset{F}_{i, j}^{(k)}(1)$ only contains the word $A_{i, j}^{(k)}$. We therefore have
\begin{equation}
\matindsbar{B}{i}{j}(1, k; \lambda, z) =  \matinds{B}{i}{j}(k; \lambda, z) = (1 - \delta_{i, j}) z^{|A_{i , j}^{(k)}|_{\metric}} R_{i, j}^{(k)}(\lambda) = \sum_{w \in   \isset{F}_{i,j}^{(k)}(1)} z^{|w|_{\metric}} \genf{R}(e_i, w; \lambda).
\end{equation}
Let $k_d \defn (-1)^{d + 1} k$. By the induction hypothesis and Lemma~\ref{lemma:first_in_terms_of_funds},
\begin{align}
\matindsbar{B}{i}{j}(d + 1, k; \lambda, z) &= \sum_{\ell}\matindsbar{B}{i}{\ell}(d, k_d; \lambda, z)\, \matinds{B}{\ell}{j}(- k_d; \lambda, z) \nonumber\\
&=  \sum_{\ell} \sum_{w \in   \isset{F}_{i,\ell}^{(k)}(d)} z^{|w|_{\metric}} \genf{R}(e_i, w; \lambda) (1 - \delta_{\ell, j}) z^{|A_{\ell , j}^{(-k_d)}|_{\metric}} R_{\ell, j}^{(-k_d)}(\lambda) \nonumber\\
&= \sum_{\ell} (1 - \delta_{\ell, j}) \sum_{w \in   \isset{F}_{i,\ell}^{(k)}(d)}  z^{|w|_{\metric} + |A_{\ell , j}^{(-k_d)}|_{\metric}} \genf{R}\left(e_i, w A_{\ell , j}^{(-k_d)} ; \lambda \right) \nonumber\\
&=  \sum_{w \in   \isset{F}_{i,j}^{(k)}(d + 1)} z^{|w|_{\metric}} \genf{R}(e_i, w; \lambda) \label{eq:BBar_induction_last_line},
\end{align}
where we have used Eq.~\eqref{eq:metric_condition} in Eq.~\eqref{eq:BBar_induction_last_line}.
\end{proof}
We are now ready to prove Propositions~\ref{prop:A_matrix_properties} and \ref{prop:G_in_terms_of_A}.
\begin{proof}[Proof of Proposition~\ref{prop:A_matrix_properties}]
In the region $(\lambda, z) \in (0, 1] \times (0, 1]$, the non-negativity of $\ismat{K}(\lambda, z)$ immediately follows from the definitions of $|\cdot|_{\metric}$ and $R_{i, j}^{(k)}(\lambda)$. For the remainder of the proof, consider the region $(\lambda, z) \in (0, 1) \times (0, 1]$. Observe that, by Eq.~\eqref{eq:B_prod_mat_def}, we have
\begin{subequations}
\begin{align}
\ismat{K}^{2d} &=  \begin{pmatrix}
\overline{\ismat{B}}(2d, 1; \lambda, z)  & \zeromat \\
\zeromat & \overline{\ismat{B}}(2d, -1; \lambda, z)
\end{pmatrix} \quad d \geq 1 \\
\ismat{K}^{2d - 1} &=  \begin{pmatrix}
\zeromat & \overline{\ismat{B}}(2d - 1, 1; \lambda, z)  \\
\overline{\ismat{B}}(2d - 1, 1; \lambda, z) & \zeromat
\end{pmatrix} \quad d \geq 1.
\end{align}
\end{subequations}
The matrix $\ismat{B}(k; \lambda, z)$ has zeroes on its diagonal and strictly positive entries elsewhere; hence it is primitive. It follows that $\overline{\ismat{B}}(d, k; \lambda, z)$ has strictly positive entries for $d \geq 2$, so $\ismat{K}(\lambda, z)$ is irreducible with period 2. Next, define the double generating function
\begin{align}
    \genf{F}_{i, j}^{(k)}(\lambda, z) &\defn \sum_{n = 0}^{\infty} E_{e_i}[z^{|W_n|_{\metric} } \,\indic(\target(W_n) = j,\ \mathfrak{K}(W_n) = k )] \lambda^n   \\
    &= \sum_{d = 1}^{\infty}\sum_{w \in \isset{F}_{i,j}^{(k)}(d) } z^{|w|_{\metric}}   \genf{S}(e_i, w; \lambda), \label{eq:Gijk_simplified}
\end{align}
where $\indic(\cdot)$ is the indicator function. Writing $\genf{S}(x, y; \lambda) = \genf{R}(x, y; \lambda)\genf{S}(y, y;\lambda)$ and noting that $\genf{S}(y, y; \lambda) \geq 1$, we have that Lemma~\ref{lemma:B_in_terms_of_R} implies
\begin{equation} \label{eq:Bbar_mat_in_terms_of_S}
\matindsbar{B}{i}{j}(d, k; \lambda, z) \leq  \sum_{w \in   \isset{F}_{i,j}^{(k)}(d)} z^{|w|_{\metric}} \genf{S}(e_i, w; \lambda)  \quad d \geq 1.
\end{equation}
Combining Eqs.~\eqref{eq:Gijk_simplified} and \eqref{eq:Bbar_mat_in_terms_of_S} and using the fact that $\lambda \in (0, 1)$ gives
\begin{equation}
   \sum_{d = 1}^{\infty} [\ismat{K}(\lambda, z)^d]_{\ell, m} \leq \max_{i, j, k} \sum_{d = 1}^{\infty} \matindsbar{B}{i}{j}(d, k; \lambda, z)\leq  \max_{i, j, k} \genf{F}_{i, j}^{(k)}(\lambda, z) \leq \frac{1}{1 - \lambda} < \infty,
\end{equation}
which shows that the spectral radius of $\ismat{K}(\lambda, z)$ is strictly less than 1.
\end{proof}

\begin{proof}[Proof of Proposition~\ref{prop:G_in_terms_of_A}]
We have
\begin{equation} \label{eq:g_S}
\genf{G}_i(\lambda, z) =  \sum_{w \in \arrowset_i} z^{|w|_{\metric}}   \genf{S}(e_i, w; \lambda)
=  \genf{S}(e_i, e_i; \lambda) + \sum_{d = 1}^{\infty}\sum_{k, j} \sum_{w \in \isset{F}_{i,j}^{(k)}(d) } z^{|w|_{\metric}}   \genf{S}(e_i, w; \lambda).
\end{equation}
Applying Lemma~\ref{lemma:B_in_terms_of_R} gives, for $d \geq 1$,
\begin{align}
 \bar{\isvec{v}}\transpose(i) \, \ismat{K}(\lambda, z)^d  \,\bar{\isvec{s}}(\lambda) &= \sum_{k} \isvec{v}\transpose(i) \, \overline{\ismat{B}}(d, k; \lambda, z) \,\isvec{s}(\lambda) \nonumber \\
&= \sum_{k} \sum_{\ell, j} \vecinds{v}{\ell}(i)\, \matindsbar{B}{\ell}{j}(d, k; \lambda, z)  \,\vecinds{s}{j}(\lambda) \nonumber \\
&= \sum_{k}   \sum_{\ell , j} \sum_{w \in   \isset{F}_{\ell,j}^{(k)}(d)} \delta_{i, \ell}  \, z^{|w|_{\metric}} \genf{R}(e_\ell, w; \lambda) \,\genf{S}(e_j, e_j; \lambda)  \nonumber \\
&= \sum_{k, j} \sum_{w \in   \isset{F}_{i,j}^{(k)}(d)} z^{|w|_{\metric}} \,\genf{S}(e_i, w; \lambda). \label{eq:vkS_final}
\end{align}
Substituting Eq.~\eqref{eq:vkS_final} into Eq.~\eqref{eq:g_S} completes the proof.
\end{proof}


\appendix

\section{Supplemental proofs}
\label{sec:appendix_proofs}

\begin{proof}[Proof of Lemma \ref{lemma:plane_switching}]

Suppose first that two different reduced words $w, w'$ are in the same equivalence class. Then there is a sequence of words $w_0=w, w_1, w_2, \dots, w_j=w'$ with $j\ge 1$  such that each $w_{i}$ is obtained from $w_{i - 1}$ by a single operation, namely an application of Eq.~\eqref{eq:groupoid_letter_relations}. We will call such an application an ``up'' (``down'') type move if it increases (decreases) the number of generating elements in the word.

 We start by showing that each local sequence of the form $w_{i - 1}\overset{\text{up}}{\to}  w_i \overset{\text{down}}{\to} w_{i + 1}$ can be edited to a new sequence $w_{i - 1}, w_{i - 1, 1}, w_{i - 1, 2} \dots  w_{i - 1, \ell}, w_{i + 1}$ such that all down moves appear to the left of all up moves. If the down move does not include the generators that have been inserted by the up move, then the up and down moves commute so we can place the down move first. There are finitely many cases when the down move includes one of the generators created by the up move. One can check the cases one-by-one to ensure that they can be edited as required; we omit the details here.

The previous statement implies that the entire sequence $w_0, w_1, w_2, \dots w_j$ can be edited to a new sequence $\bar{w}_0=w, \bar{w}_1, \bar{w}_2, \dots \bar{w}_k=w'$ such that all down moves appear to the left of all up moves, and $k\ge 1$. But this cannot happen: since $w, w'$ are both reduced, $\bar{w}_0\to \bar{w}_1$ cannot be a down move and $\bar{w}_{k-1}\to \bar{w}_k$ cannot be an up move. This shows that an arrow cannot have two different reduced representations.

To finish the proof we need to show that each arrow has a reduced representation. For a given composition of finitely many generating elements, we can use the greedy algorithm to apply down steps until we cannot do it anymore. The resulting word  cannot include a product of the form $A_{i, j}^{(k)} A_{j, \ell}^{(k)}$, which means that it must be empty or of the form Eq.~\eqref{eq:generic_reduced_word}.
\end{proof}

\begin{proof}[Proof of Lemma~\ref{lemma:intermediate_words}]

The first statement follows directly from the structure of transition probability function~\eqref{eq:full_transition_probs} and Assumptions~\ref{ass:append_A_each_step} and \ref{ass:prob_form}.

To prove \ref{lemma:intermediate_words_sources}, by \ref{lemma:intermediate_words_markov} we may assume that $w_1 = e_i$ for some $1 \leq i \leq N$.  By Assumptions~\eqref{ass:append_A_each_step} and \ref{ass:prob_form} we have that $\source(W_n) = \source(W_{n + 1})$ for all $n \geq 0$ so we must have $\source(w_2) = i$ if $e_i$ and $w_2$ are connected with a positive probability path. On the other hand, if $\source(w_2) = i$ and $w_2\neq e_i$ then by Lemma~\ref{lemma:plane_switching} we have that $w_2$ can be written in the form
\begin{equation}
    w_2 = A_{i_1, i_2}^{(k)} A_{i_2, i_3}^{(-k)} A_{i_3, i_4}^{(k)} \cdots  A_{i_d, i_{d + 1}}^{((-1)^{d + 1} k)} \text{ such that } i_\ell \neq i_{\ell + 1} \text{ for all } 1 \leq \ell \leq d \text{ and } i_1 = i,
\end{equation}
for some $d \geq 1$. The partial products of this representation provide a positive probability path from $e_i$ to $w_2$ by Assumption~\ref{ass:probs_positive}. This completes the proof of \ref{lemma:intermediate_words_sources}.

We first prove \ref{lemma:intermediate_words_intermediate} in the special case when $w_1=e_i$, $w_2=A_{i, j}^{(k)} A_{j, \ell}^{(-k)}$. Let $v_1=w_1$, $v_2$, $\dots$, $v_{n}=w_2$ be a positive probability path for the Markov chain. We need to show that there is a $0<n'<n$ with $v_{n'}=A_{i, j}^{(k)} $. Let $\alpha = \max\{i < n : |v_{i}| = 1 \}$; this is well defined since $|v_i|$ changes by at most 1 in each step. This also shows that $|v_{k}|\geq 2$ for $\alpha< k \leq n$ and that the first generator in the reduced representation of $v_k$ is $v_\alpha$. Since this first generator is $A_{i, j}^{(k)}$ for $k=n$, this implies that $v_\alpha=A_{i, j}^{(k)} $.
To prove \ref{lemma:intermediate_words_intermediate} in the general case we can iterate this argument with the help of \ref{lemma:intermediate_words_markov}.
\end{proof}

\begin{lemma} \label{lemma:M_primitive}
The matrix $\ismat{M}(\lambda)$ defined in Eq.~\eqref{eq:M_mat_def} is primitive for all $N \geq 3$ and $\lambda \in (0, 1]$.
\end{lemma}
\begin{proof}
To show that $\ismat{M}$ is primitive it is sufficient to show that there is an $n$ such that each entry of $\ismat{M}^n$ is strictly positive. Let $\Gamma(\ismat{Z})$ be the matrix obtained by replacing all the strictly positive entries of any matrix $\ismat{Z}$ with 1. By the assumption that $\lambda \in (0, 1]$, and by direct computation, we have
\begin{subequations}
\begin{align}
\Gamma(\ismat{M})_{(i, j, k)}^{(i', j', k')} &= \delta_{j, j'} \delta_{k, k'} + \delta_{j', i}(1 - \delta_{k', k}), \label{eq:M_primitive_proof_line1} \\
\Gamma(\ismat{M}^2)_{(i, j, k)}^{(i', j', k')} &= \delta_{j, j'} \delta_{k, k'} + (1 - \delta_{j, j'})(1 - \delta_{k, k'}) + \delta_{j', i}(1 - \delta_{k, k'}) + (1 - \delta_{j', i}) \delta_{k, k'}, \label{eq:M_primitive_proof_line2}\\
\Gamma(\ismat{M}^3)_{(i, j, k)}^{(i', j', k')} &= 1. \label{eq:M_primitive_proof_line3}
\end{align}
\end{subequations}
We remark that $N \geq 3$ is used between Eqs.~\eqref{eq:M_primitive_proof_line1} and \eqref{eq:M_primitive_proof_line2}.
\end{proof}

\begin{lemma} \label{lemma:sequence_converges_to_r}
For the sequence $\{\isvec{a}_n(\lambda)\}_{n \geq 0}$ defined in Eq.~\eqref{eq:a_seq_defn}, we have
\begin{equation} \label{eq:a_induction_reslt}
     (\isvec{a}_n(\lambda))_\ell \leq \sum_{m = 0}^{2^{n + 1} - 1}  P_{e_i}\left(T\left( 0, A_{i, j}^{(k)} \right) = m \right) \lambda^m  \quad n \geq 0.
\end{equation}
where $T(m, x)$ is defined in Eq.~\eqref{eq:T_stopping_time_def} and $\ell = (i, j, k)$.
\end{lemma}
\begin{proof}
The proof is by induction on $n \geq 0$. We will write $t_{i, j}^{(k)}(\alpha) \defn P_{e_i}\left(T\left( 0, A_{i, j}^{(k)} \right) = \alpha \right)$ to simplify the notation. The base case $n = 0$ is immediate since $t_{i, j}^{(k)}(1) = p_{i, j}^{(k)}$ so $(\isvec{a}_0(\lambda))_\ell = p_{i, j}^{(k)} \lambda$. By the induction hypothesis and Eq.~\eqref{eq:f_non_decreasing}, we have
\begin{subequations} \label{eq:anplus1}
\begin{align}
    (\isvec{a}_{n + 1}(\lambda))_\ell &\leq \lambda p_{i, j}^{(k)} \label{eq:an_line1}\\
    &+ \sum_{\alpha = 0}^{2^{n + 1} - 1} \sum_{m \neq i, j }   p_{i, m}^{(k)} t_{m, j}^{(k)}(\alpha) \lambda^{\alpha + 1} \label{eq:an_line2}  \\
    &+ \sum_{\alpha = 0}^{2^{n + 1} - 1} \sum_{\beta = 0}^{2^{n + 1} - 1} \sum_{m \neq i }  p_{i, m}^{(-k)} t_{m, i}^{(-k)}(\alpha) t_{i, j}^{(k)}(\beta) \lambda^{\alpha + \beta + 1}. \label{eq:an_line3}
\end{align}
\end{subequations}
We show that the right side of Eq.~\eqref{eq:anplus1} can be bounded above by
\begin{align}
\sum_{m = 0}^{2^{n + 2} - 1}  P_{e_i}\left(T\left( 0, A_{i, j}^{(k)} \right) = m \right) \lambda^m.
\end{align}
The term in \eqref{eq:an_line1} is bounded by
\begin{align}
\sum_{m = 0}^{2^{n + 2} - 1}  P_{e_i}\left(T\left( 0, A_{i, j}^{(k)} \right) = m , W_1= A_{i, j}^{(k)}\right) \lambda^m,
\end{align}
and the term in line~\eqref{eq:an_line2} is bounded by
\begin{align} \label{eq:line2_bound}
    \sum_{\alpha = 0}^{2^{n + 2} - 1} \sum_{m \neq i, j }   P_{e_i}\left(T\left( 0, A_{i, j}^{(k)} \right) = \alpha ,  W_1 = A_{i, m}^{(k)} \right)  \lambda^{\alpha} .
\end{align}
Finally, we consider line~\eqref{eq:an_line3}. By Lemma~\ref{lemma:intermediate_words}, we can write
\begin{equation} \label{eq:t_convolution}
     P_{e_m}\left(T\left( 0, A_{m, i}^{(-k)} A_{i, j}^{(k)}  \right) = \beta \right) =  \sum_{\alpha = 0}^{\beta} t_{m, i}^{(-k)}(\alpha)\, t_{i, j}^{(k)}(\beta - \alpha)
\end{equation}
with $t_{i, j}^{(k)}(0) = 0$. For non-negative numbers $c_{\alpha, \beta}$ indexed by $0 \leq \alpha \leq n$, $0 \leq \beta \leq n$, we have
\begin{equation} \label{eq:double_sum_upper_bound}
    \sum_{\alpha = 0}^{n} \sum_{\beta = 0}^{n} c_{\alpha, \beta} \leq \sum_{r = 0}^{2 n } \sum_{\alpha + \beta = r} c_{\alpha, \beta}.
\end{equation}
Therefore, applying Eq.~\eqref{eq:double_sum_upper_bound} then Eq.~\eqref{eq:t_convolution} results in
\begin{equation}
\sum_{\alpha = 0}^{2^{n + 1} - 1} \sum_{\beta = 0}^{2^{n + 1} - 1} t_{m, i}^{(-k)}(\alpha) t_{i, j}^{(k)}(\beta) \lambda^{\alpha + \beta} \leq \sum_{r = 0}^{2^{n + 2} - 2 } P_{e_m}\left(T\left( 0, A_{m, i}^{(-k)} A_{i, j}^{(k)}  \right) = r \right) \lambda^r .
\end{equation}
Hence, the quantity on line~\eqref{eq:an_line3} can be bounded above by
\begin{align} \label{eq:line3_bound}
    \sum_{\alpha = 0}^{2^{n + 2} - 1} \sum_{m \neq i, j }   P_{e_i}\left(T\left( 0, A_{i, j}^{(k)} \right) = \alpha , W_1 = A_{i, m}^{(-k)} \right)  \lambda^{\alpha},
\end{align}
which completes the induction step and hence the proof.
\end{proof}

\section*{Acknowledgments}

The authors are grateful to Steven Lalley for valuable discussions and for his assistance with the proof of Proposition~\ref{prop:transience}.  B.V.~was partially supported by the NSF award DMS-1712551.

\bibliographystyle{siam}
\bibliography{my_bib.bib}

\end{document}